\documentclass[11pt]{amsart}

\usepackage{epsf,amssymb,amsmath,overpic,amscd,psfrag,stmaryrd,times,mathrsfs, euscript,thmtools, amsbsy}

\usepackage[doi=false]{biblatex}
\usepackage{caption}
\usepackage{subcaption}

\usepackage[all]{xy}
\usepackage{hyperref, stmaryrd}
\usepackage[dvipsnames]{xcolor}
\usepackage{enumitem}
\usepackage{tikz,tikz-cd}

\tikzcdset{scale cd/.style={every label/.append style={scale=#1},
    cells={nodes={scale=#1}}}}

\usepackage{todonotes}

\input{biblatex.cfg}

\newtheorem{theorem}{Theorem}[section]

\newtheorem{claim}[theorem]{Claim}

\newtheorem{question}[theorem]{Question}
\newtheorem{corollary}[theorem]{Corollary}
\newtheorem{lemma}[theorem]{Lemma}
\newtheorem{proposition}[theorem]{Proposition}

\theoremstyle{definition}
\newtheorem{definition}[theorem]{Definition}
\newtheorem{example}[theorem]{Example}

\newtheorem{remark}[theorem]{Remark}
\newtheorem{assumption}[theorem]{Assumption}

\numberwithin{equation}{subsection}

\DeclareMathAlphabet{\mathpgoth}{OT1}{pgoth}{m}{n}

\DeclareMathAlphabet{\mathpzc}{OT1}{pzc}{m}{it}

\newcommand{\R}{\mathbb{R}}
\newcommand{\Z}{\mathbb{Z}}

\newcommand{\V}{\mathbb{V}}

\newcommand{\Kur}{\mathscr{K}}
\newcommand{\be}{\begin{enumerate}}
\newcommand{\ee}{\end{enumerate}}
\newcommand{\op}{\operatorname}
\newcommand{\bs}{\boldsymbol}

\newcommand{\arr}{\overrightarrow}

\newcommand{\rr}{\mathbb R}

\newcommand{\hh}{\mathbb H}

\DeclareMathOperator{\coker}{coker}

\DeclareMathOperator{\ind}{Ind}

\DeclareMathOperator{\crit}{Crit}

\newcommand{\h}{\bm H} 

\newcommand{\cleanIntersect}{\Cap}

\newcommand{\grad}[2]{\text{grad}_{({#1},{#2})}}

\newcommand{\hess}[1]{\text{Hess}_{#1}}

\newcommand{\tildeM}{\widetilde{\mathcal M}}
\newcommand{\M}{\mathcal M}
\newcommand{\morsedbar}{\mathcal L}

\newcommand{\closeToBreaking}[3]{\mathcal N^{#1}({#2},{#3})}
\newcommand{\closeToBreakings}[2]{\mathcal N^{#1}(\compactSubset_{\iT{#2}})}

\newcommand{\trans}[1]{\tau_{#1}}

\newcommand{\bump}[2]{{\beta}_{#1}^{#2}}
\newcommand{\bPlus}[1]{\bump{+}{#1}}
\newcommand{\bMinus}[1]{\bump{-}{#1}}

\newcommand{\umap}[2]{u_{#1, #2}}
\newcommand{\uPlus}[1]{\umap{+}{#1}}
\newcommand{\uMinus}[1]{\umap{-}{#1}}

\newcommand{\Perturb}{\psi}
\newcommand{\perturb}[2]{{\psi}_{#1, #2}}
\newcommand{\pPlus}[1]{\perturb{+}{#1}}
\newcommand{\pMinus}[1]{\perturb{-}{#1}}

\newcommand{\iT}[1]{#1}

\newcommand{\indexTuple}[1]{(\MakeLowercase{#1}_1,\MakeLowercase{#1}_2,\dots, \MakeLowercase{#1}_{n(\iT{#1})})}

\newcommand{\SIC}[1]{\arr{#1}}

\newcommand{\contraction}[2]{{#1}/{#2}}

\newcommand{\compactSubset}{\mathtt K}
\newcommand{\openSubset}{U}
\newcommand{\obstructionSectionS}{\mathfrak{s}}
\newcommand{\obstructionSection}[1]{\mathfrak{s}_{\iT{#1}}}
\newcommand{\obstructionBundle}{\mathcal O}
\newcommand{\productBundle}[1]{\mathbb{O}_{\iT{#1}}}
\newcommand{\obstructionBundlePM}{\obstructionBundle_{+-}}

\newcommand{\parametrizedModF}[2]{\widetilde{\mathcal M}(#1,#2)}
\newcommand{\modF}[2]{{\mathcal M}(#1,#2)}
\newcommand{\gluableSet}{\mathcal V}
\newcommand{\source}{\mathtt{s}}
\newcommand{\target}{\mathtt{t}}

\newcommand{\kChart}[1]{V_{\iT{#1}}}
\newcommand{\chartData}[1]{\mathcal{C}_{\iT{#1}}}
\newcommand{\productModuliSpaces}[2]{{\mathbb U}_{\iT{#1}}^{#2}}

\newcommand{\interiorKChart}[1]{V_{#1}}

\newcommand{\breakingBase}[1]{B_{{#1}}}

\newcommand{\breakingBundle}[1]{B^\sharp_{{#1}}}

\newcommand{\lOc}{\widetilde{\mathcal{O}}}
\newcommand{\lavg}{f}

\newcommand{\kur}{\mathcal K}
\newcommand{\kuranishiSection}[1]{\sigma_{\iT{#1}}}
\newcommand{\kuranishiSectionBold}{\boldsymbol{\sigma}}

\newcommand{\productCharts}[1]{\mathbb{V}_{\iT{#1}}}
\newcommand{\productChartsIntersection}[2]{\mathbb{V}_{\iT{#1}\iT{#2}}}

\newcommand{\lN}{\mathcal N}

\newcommand{\lO}{\mathcal O}
\newcommand{\lW}{\mathcal W}
\newcommand{\lH}{\mathcal H}

\newcommand{\dPlus}{D_+}
\newcommand{\dMinus}{D_-}

\makeatletter
\newcommand{\labitem}[2]{%
\def\@itemlabel{#1}
\item
\def\@currentlabel{#1}\label{#2}}
\makeatother

\begin{document}
\title[Computable, obstructed Morse homology for clean intersections]
{Computable, obstructed Morse homology for clean intersections}

\author{Erkao Bao}
\address{School of Mathematics, University of Minnesota, Minneapolis, MN 55455}
\email{bao@umn.edu} \urladdr{https://erkaobao.github.io/math/}

\author{Ke Zhu}
\address{Department of Mathematics and Statistics, Minnesota State University Mankato, Mankato, MN 56001}
\email{ke.zhu@mnsu.edu}
\keywords{Morse homology, Morse-Smale condition, transversality, clean intersection, obstruction bundle gluing, semi-global Kuranishi structure, iterated gluing = simultaneous gluing}

\subjclass[2010]{Primary 53D40; Secondary 53D58.}
\thanks{Erkao Bao is supported by NSF Grants DMS-2404529.}

\begin{abstract} 
In this paper, we develop a method to compute the Morse homology of a manifold when descending manifolds and ascending manifolds intersect {\em cleanly}, but not necessarily transversely.

While obstruction bundle gluing \cite{hutchings2009gluing} is a computable tool to handle non-transverse intersections, it has only been developed for specific cases. In contrast, most virtual techniques apply to general cases but lack computational efficiency. 
To address this, we construct {\em minimal semi-global Kuranishi structures} for the moduli spaces of Morse trajectories, which generalize obstruction bundle gluing while maintaining its computability feature, in the spirit of semi-global Kuranishi structures \cite{bao2015semi}. Through this construction, we obtain iterated gluing equals simultaneous gluing.
\end{abstract}

\maketitle

\setcounter{tocdepth}{2}


\section{Introduction}
In Morse homology or Floer homology, it is typical to perturb the metric or the almost complex structure so that the moduli space of Morse trajectories or pseudo-holomorphic curves is transversely cut out. However, such perturbations are sometimes  insufficient to achieve transversality. In the Floer case, this insufficiency is due to the presence of bubbles or multiple covers of pseudo-holomorphic curves. In the Morse case, when studying equivariant Morse homology, it is well-known that equivariant transversality cannot always be achieved when the Morse function is unstable (see \cite{bao2024morse}).

To address this issue, several abstract perturbation schemes have been developed, including: Kuranishi structures \cite{fukaya1999arnold, fooo, ishikawa2021smooth}, implicit atlases \cite{pardon2019contact}, Kuranishi atlases \cite{mcduff2018fundamental}, polyfolds \cite{hofer2007general}, semi-global Kuranishi structures \cite{bao2015semi}, global Kuranishi charts \cite{abouzaid2021arnold}, D-manifolds \cite{joyce2012d-manifold}, and the virtual fundamental cycles \cite{li1998virtual, liu1998floer, ruan1999virtual}. 
It is fairly straightforward to construct a global Kuranishi chart for Morse homology, for instance, by turning on perturbations of several vector fields near each critical point. However, these perturbations often excessively thicken the moduli spaces, making computation infeasible. On the other hand, obstruction bundle gluing, originally defined in \cite{hutchings2009gluing}, has the advantage of being computable, though it was developed for a special case. We identify the condition for generalizing obstruction bundle gluing as clean intersection and extend the obstruction bundle gluing in the spirit of a semi-global Kuranishi structure, while maintaining its computability.

When the moduli spaces of Morse trajectories are transversely cut out, the compactified moduli space is a smooth manifold with corners. Such a corner structure is constructed in \cite{wehrheim2012smooth}. The boundary/corner charts are given by gluing broken trajectories consisting of at least three smooth pieces. The corner structure requires the compatibility condition: iterated gluing = simultaneous gluing. As the gluing map is defined in a non-canonical way, simultaneous gluing and iterated gluing are generally not the same, but their difference vanishes as the gluing parameters go to infinity. In \cite{wehrheim2012smooth} the gluing map was meticulously constructed so that simultaneous gluing = iterated gluing. Kenji Fukaya suggested that one could modify the gluing maps by averaging them using the idea of the center of mass in Riemannian geometry. This idea is very likely to lead to a very simple alternative proof of this result. When the moduli spaces are not transverse, not every broken trajectory can be glued into a smooth trajectory; usual gluing is upgraded into obstruction bundle gluing. The simultaneous gluing $\neq$ iterated gluing problem persists and becomes more chanllenging as the dimensions of the charts vary. The issue is inconvenient but not unavoidable if one only cares about dimension 0 or 1 moduli spaces (see \cite{bao2015semi}). The averaging method can be hard to adapt to the obstruction bundle gluing case. We shrink the charts so that different charts only intersect only when the broken curves in one chart are eitheralso  broken curves of the other chart or arise from gluing some broken curves in the other chart. We then modify the gluing maps inductively.

When the moduli space $\M$ of Morse trajectories is cleanly cut out, $\M$ is a smooth manifold, and over it, there exists an obstruction bundle $\obstructionBundle$, which is the cokernel bundle. In this case, there is no need to construct a thickened moduli space as is usually done in Kuranishi structures. Our goal is to perturb the moduli space so that it is cut out transversely. By choosing a section $\sigma$ of the obstruction bundle such that $\sigma$  intersects the zero section of $\obstructionBundle$ transversely, we replace a large open subset of the moduli space $\M$ with $\sigma^{-1}(0)$. This is sufficient for the ``interior" of $\M$. For the boundary of $\M$, the perturbed moduli space needs to exhibit some boundary/corner structure, provided by the products of the relevant lower energy strata. The boundary chart is obtained by gluing broken trajectories. In particular, unlike the transverse case, not all broken trajectories can be glued to smooth ones. 

Indeed, on the space of broken trajectories, together with the space of gluing parameters, there is a bundle formed by taking the direct sum of the obstruction bundles of the moduli spaces of smooth pieces. Over this bundle, there exists a section, called the obstruction section $\mathfrak s$, such that $\mathfrak s^{-1}(0)$ is the space of broken trajectories (and gluing parameters) that can be glued to smooth curves. In other words, through the inverse of the gluing map, the curves that are close to breaking are mapped to the boundary charts (spaces of broken trajectories) as $\obstructionSectionS^{-1}(0)$. The obstruction section, originally constructed by \cite{hutchings2009gluing}, has the advantage that its leading terms can be explicitly expressed. This plays an important role since we perturb $\obstructionSectionS$ to obtain a transverse section in the boundary chart. The obstruction section $\obstructionSectionS$ must satisfy some compatibility conditions if there are multiple boundary charts. Specifically:

Suppose there is a broken trajectory consisting of three trajectories $u_1$ and $u_2$, and $u_3$. Suppose that the two trajectories $u_1, u_2$ glue with the gluing parameter $T_1$ to a curve $u_4 = G_{12}(u_1, T_1, u_2)$, where $G_*$ denotes the gluing map. This means that $\obstructionSectionS_{12}(u_1, T_1, u_2)=0$, where $\obstructionSectionS_{12}$ is the obstruction section on the space of broken trajectories containing $(u_1, T_1, u_2)$. Suppose also that $u_4$ and $u_3$ glue with the gluing parameter $T_2$ to a smooth trajectory $u_6 = G_{43}(u_4, T_2, u_3)$, which implies that $\obstructionSectionS_{43}(u_4, T_2, u_3) = 0$.
Then, one can ask whether $(u_1, u_2, u_3)$ can be simultaneously glued with gluing parameters $T_1$ and $T_2$, i.e., whether $\obstructionSectionS_{123}(u_1, T_1, u_2, T_2, u_3) = 0$, and if so, whether the glued curve $G_{123}(u_1, T_1, u_2, T_2, u_3)$ is the same as the iteratively glued curve $G_{43}(u_4, T_2, u_3) = G_{43}(G_{12}(u_1, u_2, T_1), T_2, u_3)$.

To answer this question, we interpret the obstruction section $\obstructionSectionS$ in the Kuranishi setting. A Kuranishi chart is the data $(V,E,\morsedbar,\phi)$, where $V$ is the thickened moduli space, $E \to V$ is a vector bundle, $\morsedbar$ is a section of $E \to V$ (in the contex of morse homology, $\morsedbar$ is given by \eqref{eqn: morsedbar}), and $\phi: \morsedbar^{-1}(0) \to \M$ is a homeomorphism onto its image. Suppose $(V_\pm, E_\pm, \morsedbar_\pm,\phi_\pm)$ are two other charts of lower energy that can be glued to elements in $V$. Indeed, we have a gluing map $(G,G^\sharp)$ which is a bundle isomorphism:
\begin{equation*}
\begin{tikzcd}
    E_+\oplus E_- \arrow[r, "G^\sharp"] \arrow[d] & E \arrow[d] \\
    V_+ \times V_- \times [R,\infty) \arrow[r, "G"] \arrow[u, bend left, "\morsedbar_+ \times \morsedbar_-"]&  V \arrow[u, bend left, "\morsedbar"],
\end{tikzcd}
\end{equation*}
where $R>0$ is a large number. The obstruction section $\obstructionSectionS$ can be equivalently defined to be $\morsedbar_+ \times \morsedbar_- - (G^\sharp)^{-1} \circ \morsedbar \circ G$. In the case of clean intersection, one can choose $E_\pm$ to be the obstruction bundles, $V_\pm$ to be open subsets in $\M_\pm$, and $\morsedbar_\pm = 0$. With this definition, as long as the gluing maps $(G, G^\sharp)$ satisfy the condition iterated gluing equals simultaneous gluing, the obstruction sections are compatible:
\begin{align*}
    & \obstructionSectionS_{123}(u_1, T_1, u_2, T_2, u_3) \\
    = &  (G_{123}^\sharp)^{-1} \circ \morsedbar \circ G_{123} (u_1, T_1, u_2, T_2, u_3) \\
    = &  (G_{12}^\sharp \times \op{id})^{-1} \circ (G_{43}^\sharp)^{-1} \circ \morsedbar \circ G_{43} \circ (G_{12}\times \op{id}) (u_1, T_1, u_2, T_2, u_3) \\
    = &  (G_{12}^\sharp \times \op{id})^{-1} \circ (G_{43}^\sharp)^{-1} \circ \morsedbar \circ G_{43} (u_4, T_2, u_3) \\
    = & (G_{12}^\sharp \times \op{id})^{-1} \circ \obstructionSectionS_{43} (u_4, T_2, u_3).
\end{align*}
If $\obstructionSectionS_{43} (u_4, T_2, u_3) = 0$, then $\obstructionSectionS_{123}(u_1, T_1, u_2, T_2, u_3) = 0.$

We provide the simplest non-trivial calculation of the Morse homology of an upright torus as a proof of concept (Example~\ref{example: upright torus}). The calculation relies only on the information of the moduli spaces, obstruction bundles, and obstruction sections. In particular, it does not require extra knowledge of the geometry of the manifolds.

\subsection{Outline}

In Section~\ref{section: clean intersection}, we recall the definition of clean intersection and examine its basic properties. We show that, given any regular (which is generic) $1$-parameter family of Riemannian metrics, for any moment in the family, index = 0 moduli spaces of Morse trajectories are cleanly cut out.

Section~\ref{section: gluing} focuses on obstruction bundle gluing within the clean intersection setting. We describe the gluing of broken trajectories composed of two pieces and introduce the obstruction section $\obstructionSectionS$. Here, we also state the gluing theorem for this setting and outline the strategy to perturb the obstruction section to subsequently perturb moduli spaces of Morse trajectories. This approach is illustrated by computing the Morse homology of an upright torus.

In Section~\ref{section: proof of theorem}, we prove the gluing theorem presented in Section~\ref{section: gluing} and delve into the simultaneous gluing of broken trajectories with multiple (two or more) pieces. 

Section~\ref{section: definition of minimal semi-global} introduces minimal semi-global Kuranishi structures for clean intersections and their perturbation sections. Using these, we perturb the moduli spaces.

In Section~\ref{section: morse homology}, we discuss the orientation of the perturbed moduli spaces and revisit the concepts of Morse homology.

Finally, in Section~\ref{section: construction of semi-global}, we construct a minimal semi-global Kuranishi structure. During this process, we modify the gluing maps to ensure that iterated gluing equals simultaneous gluing. We also construct a perturbation of the moduli space by choosing a perturbation section.

\section*{Acknowledgement}
The first author thanks Conan Leung for hosting him at the Institute of Mathematical Sciences, Chinese University of Hong Kong, where part of this paper was completed. The second author thanks the School of Mathematics, University of Minnesota, for providing an excellent research environment during his sabbatical year visit. He also thanks Kaoru Ono for helpful discussions.

\section{Clean Intersection and Obstruction Bundles}\label{section: clean intersection}
A more general concept than transverse intersection is the notion of clean intersection.

\begin{definition}[Clean intersection]
Two submanifolds $X$ and $Y$ of $Z$ are said to intersect cleanly, denoted as $X \cleanIntersect Y$, if $X \cap Y$ is a submanifold of $Z$ and for any $z \in X \cap Y$, $T_z(X \cap Y) = T_z X \cap T_z Y$.
\end{definition}

It is clear that transverse intersection implies clean intersection.

\begin{example}
Let $f: \R \to \R$ be the function $f(x) = x^2$. The graph of $f$ does not intersect $\R \oplus \{0\}$ cleanly.
\end{example}

\begin{example}
If the manifold $X$ is a submanifold of a manifold $Y$, and $Y$ is a submanifold of a manifold $Z$, then $X$ intersects $Y$ cleanly (in $Z$).
\end{example}

More generally,
\begin{example}\label{example:transverse-intersection}
Suppose that the manifolds $X$ and $Y$ intersect cleanly in $W$ (for instance, $X \pitchfork_W Y$), and $W$ is a submanifold of a manifold $Z$. 
Then $X$ and $Y$ intersect cleanly in $Z$.
If the codimension of $W$ inside $Z$ is $\geq 1$, then $X$ and $Y$ do not intersect transversely in $Z$.
\end{example}

\begin{proposition}[\cite{hormander2007theanalysis}]\label{prop:local-coordinate}
Suppose $X$, $Y$, and $X \cap Y$ are submanifolds of $Z$. The following two statements are equivalent:
\begin{enumerate}
    \item $X$ and $Y$ intersect cleanly.
    \item For any $z \in X \cap Y$, there exists a coordinate chart around $z$ with respect to which $X$ and $Y$ are linear subspaces.
\end{enumerate}
\end{proposition}

Unlike transverse intersections, linear algebra is not sufficient to check for clean intersections. Instead, we have the following lemma.

\begin{lemma}\label{lemma:clean-intersection-criterion}
Suppose $X$ and $Y$ are two submanifolds of a manifold $Z$. Assume that $\dim(T_z X \cap T_z Y)$ locally constant in $z \in X \cap Y$. Then the following two statements are equivalent:
\begin{enumerate}
    \item $X$ and $Y$ intersect cleanly.
    \item For each $z \in X \cap Y$ and a submanifold $N$ of $Z$ containing $z$ with codimension $\dim X - \dim(T_z X \cap T_z Y)$ such that:
    \begin{enumerate}
        \item $T_z X + T_z N = T_z Z$.
        \item $Y \cap U \subset N$ for some neighborhood $U \subset Z$ of $z$.
    \end{enumerate}
    Then, one can shrink the neighborhood $U$ if necessary, so that $X \cap Y \cap U = X \cap N \cap U$.
\end{enumerate}
\end{lemma}

\begin{proof}
The submanifold $N$ can be constructed as follows:
This construction is carried out in a small neighborhood of $z$. Let $F$ be a subbundle of $TZ|_Y$ such that $F_z \oplus (T_z X + T_z Y) = T_z Z$ for all $z \in X \cap Y$. Define $N$ as a submanifold of $Z$ in a neighborhood of $Y$ such that $T_y N = F_y \oplus T_y Y$ for any $y \in Y$. $N$ can be constructed by taking the exponential map: $F \to Z$ with respect to some Riemannian metric. Then, $Y$ is a submanifold of $N$, and $X \pitchfork_Z N$.

Suppose that there exists a small neighborhood $U \subset Z$ of $Y$ such that $X \cap Y = X \cap N \cap U$. Then, since $X$ intersects $N \cap U$ transversely, $X \cap Y$ is a submanifold of $Z$. The statement that $T_z (X \cap Y) = T_z X \cap T_z Y$ for all $z \in X \cap Y$ follows from the fact that $T_z(X \cap Y) \subset T_z X \cap T_z Y$ and $\dim T_z(X \cap Y) = \dim T_z(X \cap N) = \dim X + \dim N - \dim Z = \dim T_z (X \cap Y)$.

Conversely, suppose $X$ and $Y$ intersect cleanly. For any $z \in X \cap Y, $ let $U \subset Z$ be a small neighborhood of $z$. Since $Y \subset N$, we have $X \cap Y \cap U \subset X \cap N \cap U$. Note that $\dim (X \cap Y \cap U) = \dim T_z (X \cap Y) = \dim (T_z X \cap T_z Y) = \dim T_z X + \dim T_z Y - \dim (T_z X + T_z Y) = \dim X + \dim N - \dim Z = \dim (X \cap N \cap U)$. Hence, the two manifolds $X \cap Y \cap U$ and $X \cap N \cap U$ coincide if $U$ is a sufficiently small neighborhood of $z$. 
\end{proof}

Suppose $M^n$ is a closed smooth manifold, $f: M \to \R$ is a Morse function, and $g$ is a Riemannian metric on $M$. 

For any $p, q \in \crit{f}$, consider the Banach manifold
\begin{equation}\label{eqn: tilde B}
\widetilde{\mathcal B} = \widetilde{\mathcal B}(p,q)=  \{u \in L_1^2(\R, M) ~|~ \lim_{s \to -\infty} u(s) = p, \lim_{s \to \infty} u(s) = q\},
\end{equation}
and the Banach bundle $\mathcal E \to \widetilde{\mathcal B}$ defined by
$\mathcal E_u = L^2(u^* TM).$
Here, we define $L_1^2(\R, M)$ simply by choosing a smooth embedding of $M$ into $\R^{N}$ for some sufficiently large integer $N$, and defining $L_1^2(\R, M)$ as the obvious subspace of $L_1^2(\R, \R^N)$.
Consider the section 
\begin{equation}\label{eqn: morsedbar}
\morsedbar = \frac{d}{ds} + \grad{f}{g}: \widetilde{\mathcal B} \to \mathcal E,
\end{equation}
where $\grad{f}{g}$ is the gradient vector field of $f$ with respect to $g$.
Let $\mathcal B = \mathcal B(p,q) = \widetilde{\mathcal B}(p,q)/\R$ with the $\R$ acting on the domains as translation. 
The $\R$-action induces an identification of the fibers of $\mathcal E$, and we denote the obtained vector bundle with identified fibers as $\mathcal E \to \mathcal B$ by abusing the notation.

The moduli space $\parametrizedModF{p}{q}$ is defined to be $\morsedbar^{-1}(0)$.
We define $\modF{p}{q} = \parametrizedModF{p}{q}/\R$.
For any $u \in \widetilde{\mathcal B}$, we denote by $D_u$ the linearized operator of $\morsedbar$ at $u$, i.e., $D_u = \Pi \circ d_u \morsedbar$, where $\Pi: T_{(u,\morsedbar(u))} \mathcal E = \mathcal E_u \oplus T_u \widetilde{\mathcal B} \to \mathcal E_u$ is the projection to the fiber with respect to the $L^2$-inner product induced by a metric $g_0$ on $M$. Note that $g_0$ does not have to agree with $g$. Later in Section~\ref{section: pregluing}, we choose $g_0$ be a flat metric around each critical points of $f$. Let $\nabla$ the Levi-Civita connection with respect to $g_0$. Then we have the explicit formula for $D_u$ as follows: for any $\xi \in C^\infty(u^* TM)$,
\[
D_u \xi = \nabla_s \xi + \hess{f}(u) \xi,
\]
where $\hess{f}(u)\xi(s) = \nabla_\xi \grad{f}{g}(u(s))$ is the Hessian of $f$.
Let 
\[
D_{u}^* = - \nabla_s  + \hess{f}: L_1^2(u^* TM) \to L^2(u^* TM)
\] be the adjoint operator of $D_u$ with respect to the $L^2$-norm. 
For any critical point $p\in \crit{f}$, we denote by $\ind p$ the Morse index of $p$, i.e., the dimension of the negative eigenspace of the Hessian of $f$ at $p$.
The following proposition is well-known. See for instance Theorem 3.2 in \cite{salamonMorse}.
\begin{proposition}
The operator $D_u$ is Fredholm for any $u \in \parametrizedModF{p}{q}$, and the Fredholm index $\op{ind}u := \dim \ker D_u - \dim \coker D_u = \ind p - \ind q$.
\end{proposition}

In the definition of Morse homology, one usually assumes the Morse-Smale condition.

\begin{definition}[Morse-Smale]
The pair $(f,g)$ is said to be Morse-Smale if for all critical points $p,q$, $\parametrizedModF{p}{q}$ are transversely cut out, i.e., $D_u$ is surjective for all $u \in \parametrizedModF{p}{q}$.
\end{definition}

The Morse-Smale condition is equivalent to the condition that all stable manifolds and unstable manifolds intersect transversely according to the proof of Theorem 3.3 in \cite{salamonMorse}.

Fix a Morse function, the Smale theorem says that a generic metric $g$ makes the pair $(f,g)$ Morse-Smale. However, there are cases when we cannot perturb $g$ such as the equivariant case, and for the purpose of computation, we also do not want to perturb $g$. We relax the Morse-Smale condition to the following clean intersection condition.

For any $u \in \parametrizedModF{p}{q}$, let $\mathcal U \subset \widetilde{\mathcal B}$ be a small neighborhood of $u$ and $E \to \mathcal U$ be a finite-dimensional subbundle of $\mathcal E|_{\mathcal U} \to \mathcal U$, such that $\morsedbar$ is transverse to $E$, i.e., for any $u \in \mathcal U$, $\op{im} D_u + E_u = \mathcal E_u$. 
By the implicit function theorem, $V := \morsedbar^{-1}(E) \subset \mathcal U$ is a smooth manifold of dimension $\ind p - \ind q + \op{rank} E$.
Now $\morsedbar$ restricts to a section $\morsedbar|_V: V \to E|_V.$

\begin{definition}
We say $\parametrizedModF{p}{q}$ or $\modF{p}{q}$ is cleanly cut out if, for any $u \in \parametrizedModF{p}{q}$, 
there exists an open set $u \in \mathcal U \subset \parametrizedModF{p}{q}$, and $E\to \mathcal U$ as above, such that $\mathcal L|_V$ intersects the zero section of $E|_V$ cleanly.
\end{definition}

\begin{lemma}
The moduli space $\parametrizedModF{p}{q}$ is cleanly cut out if and only if the unstable submanifold $W^u(p)$ of $p$ and the stable manifold $W^s(q)$ of $q$ intersect cleanly.
\end{lemma}

\begin{proof}
Suppose that the unstable submanifold $W^u = W^u(p)$ intersects the stable manifold $W^s = W^s(q)$ cleanly. For any $u \in \parametrizedModF{p}{q}$, we show that $\morsedbar$ intersects the zero section cleanly near $u$.

Let $m_0 = u(0) \in W^s \cap W^u$, and let $U \subset M$ be an open neighborhood of $m_0$. Let $f_1, \dots , f_\ell$ be $\ell$ smooth vector fields supported in $U$ such that $\{f_1(m_0), \dots, f_\ell(m_0)\}$ are linearly independent, and 
\begin{equation}
    \op{span}\{f_1(m_0), \dots, f_\ell(m_0)\} \oplus (T_{m_0}W^u + T_{m_0} W^s) = T_{m_0} M.
\end{equation}

Let $\mathcal U \subset \widetilde{\mathcal B}$ be a neighborhood of $u$. We now define sections $\tilde f_i$ of $\mathcal E|_{\mathcal U}$. For any $v \in \mathcal U$, $\tilde f_i(v) \in L^2(v^* TM)$ such that $\tilde f_i(v)(s) = f_i(v(s))$. Let $\mathcal N = \op{span}\{\tilde f_1, \dots, \tilde f_k\}$ be a subbundle of $\mathcal E|_{\mathcal U}$.

\begin{claim}
$\morsedbar$ is transverse to $\mathcal N$ over $\mathcal U$, and $\morsedbar^{-1} (\mathcal N) = \morsedbar^{-1}(0) \cap \mathcal U$.
\end{claim}

The claim finishes the proof that the moduli space $\parametrizedModF{p}{q}$ is cleanly cut out.

\begin{proof}[Proof of claim]
Suppose that $\morsedbar$ is not transverse to $\mathcal N$ at $u$. Then there exists $0 \neq \eta \in \ker D^*_u$ such that $\eta \perp \mathcal N_u$. From the proof of Theorem 3.3 in \cite{salamonMorse}, we know $\eta(0) \notin T_{m_0} W^s + T_{m_0} W^u$, and hence $\langle \eta(0), f_i(m_0) \rangle \neq 0$ for some $i \in \{1, \dots, \ell\}$. This implies $\eta \not\perp \tilde f_i$ if $U$ is sufficiently small, which leads to a contradiction.

Now we show $\morsedbar^{-1} (\mathcal N) \subset \morsedbar^{-1}(0) \cap \mathcal U$. The set $\morsedbar^{-1}(\mathcal N)$ consists of $\gamma \in \mathcal B$ that satisfy 
\begin{equation} \label{eqn:thickened-ode}
    \morsedbar(\gamma)(s) = \sum_{i = 1}^\ell c_i f_i(\gamma(s)),
\end{equation}
where $c_i$ are constants. Let $W^s_{\mathcal N}$ be the thickened stable submanifold defined by 
\[
W^s_{\mathcal N} = \{x \in M ~|~ \lim_{t \to \infty} \gamma(t) = q, \gamma \text{ satisfies } \eqref{eqn:thickened-ode} \text{ and } \gamma(0) = x\}.
\]
Then $W^s \subset W^s_{\mathcal N}$, and $W^s_{\mathcal N} \cap U'$ is a smooth submanifold of dimension $n - \ind q + \ell$, where $U' \subset U$ is a neighborhood of $m_0$ such that $f_1(m), \dots, f_\ell(m)$ are linearly independent for all $m \in U'$. Moreover, $T_{m_0} W^u + T_{m_0} W^s_{\mathcal N} = T_{m_0} M$. Since $W^s$ intersects $W^u$ cleanly (by Lemma~\ref{lemma:clean-intersection-criterion}), $W^u \cap (W^s_{\mathcal N} \cap U') = W^u \cap (W^s \cap U')$. This implies $\morsedbar^{-1} (\mathcal N) \subset \morsedbar^{-1}(0) \cap \mathcal U$.
\end{proof}

The other direction of the statement--that the moduli space $\parametrizedModF{p}{q}$ being cleanly cut out implies the stable and unstable manifolds intersect cleanly--can be proved analogously. We leave this proof to the readers.
\end{proof}

\begin{remark}[Terminology]
    Instead of saying $\morsedbar$ is transverse to a neighborhood of the zero section of the sub-bundle $\mathcal N \subset \mathcal E$, we simply say $\morsedbar$ is transverse to $\mathcal N$.
\end{remark} 

In this paper, we assume the following:
\begin{assumption}
All moduli spaces of Morse trajectories are cleanly cut out.
\end{assumption}

We have the following example.
Let $\{g_\tau\}_{\tau \in [0,1]}$ be a 1-parameter family of Riemannian metrics. Consider the moduli space $\M = \coprod_{\tau \in [0,1]}\widetilde \M_\tau$, where $\widetilde{\M}_\tau = \widetilde{\M}_{f, g_\tau}(p,q)$ for some $p,q \in \crit{f}$.
Consider the smooth Banach bundle $\mathcal E \to \widetilde{\mathcal B}$, 
where $\widetilde{\mathcal B} = \coprod_\tau \widetilde{\mathcal B}_\tau$ and $\mathcal E = \coprod_\tau \mathcal E_\tau$.
Let $\morsedbar: \widetilde{\mathcal B} \to \mathcal E$ be the section defined by $\morsedbar|_{\widetilde{\mathcal B}_\tau} = \morsedbar_\tau: \widetilde{\mathcal B}_\tau \to \mathcal E_\tau$, where $\morsedbar_\tau = \frac{d}{ds} + \grad{f}{g_\tau}$ and $\grad{f}{g_\tau}$ is the gradient vector field of $f$ with respect to the metric $g_\tau$. 

Suppose $\ind p = \ind q$. 
We say the family $\{(f, g_\tau)\}_\tau$ is {\em regular} if the section $\morsedbar:\widetilde{\mathcal B} \to \mathcal E$ is transverse to the zero section. 

\begin{lemma} \label{lemma: generic family}
Suppose that $\{(f, g_\tau)\}_\tau$ is regular and that $\ind p = \ind q$. Then the moduli space $\widetilde{\M}_\tau$ is cleanly cut out for each $\tau$.
\end{lemma}

\begin{proof}
The regularity of the family $\{(f, g_\tau)\}_\tau$ means that for any $u \in \widetilde{\M}_\tau$, the operator $\mathbb D_u := \Pi_u \circ d \morsedbar|_{T_u \widetilde{\mathcal B}}: T_u \widetilde{\mathcal B} \to \mathcal E_u$ is surjective, where $\mathcal E_u = (\mathcal E_\tau)_u$ and $\Pi_u: T_{(u, 0)} \mathcal E \simeq  T_u \widetilde{\mathcal B} \oplus \mathcal E_u \to \mathcal E_u$ is the projection map. 
Note that $T_u \widetilde{\mathcal B} = T_u \widetilde{\mathcal B}_\tau \oplus \R\langle \frac{\partial}{\partial \tau} \rangle$, and $\mathbb D_u (\xi_u \oplus a \frac{\partial}{\partial \tau}) = (D_\tau)_u \xi_u + a Y_u$, where $D_\tau$ is the linearization of $\morsedbar_\tau$, $a \in \R$, and $Y_u = \frac{d}\nabla_\tau \grad{f}{g_\tau}(u)$.

Suppose at some $\tau = \tau_0$ and $u = u_0 \in \widetilde{\M}_\tau$, $(D_\tau)_u : T_u \widetilde{\mathcal B} \to \mathcal E_u$ is not surjective; otherwise, there is nothing to prove. (At the point $(u_0, \tau_0)$, the projection $\widetilde \M \to [0,1]$ is not a submersion.)

We first choose a finite-dimensional reduction. 
Let $\mathcal U \subset \widetilde{\mathcal B}_\tau$ be a neighborhood of $u$, and let $E \to \mathcal U$ be a finite-dimensional subbundle of $\mathcal E_\tau|_{\mathcal U} \to \mathcal U$ such that $\morsedbar_\tau$ is transverse to $E$. Then we restrict to $V := \morsedbar_\tau^{-1}(E) \subset \mathcal U$ implicitly. It is easy to see that restricting $\morsedbar_\tau: \widetilde{\mathcal B}_\tau \to \mathcal E_\tau$ to $(\morsedbar_\tau)|_V: V \to E|_V$ does not change $\dim \ker (D_\tau)_u$, $\dim \coker (D_\tau)_u$, and hence $\ind (D_\tau)_u$. We further assume $Y$ is a section of $E$.

It is clear that $Y_u \neq 0$, and $\dim \ker (D_\tau)_u = \dim \coker (D_\tau)_u = 1$.

Let $N$ be the subbundle of $E$ (more precisely, $E|_V$) spanned by $Y$.
Then $\morsedbar_\tau$ is transverse to $N$, and hence $\morsedbar_\tau^{-1}(N)$ is a 1-dimensional manifold, and hence diffeomorphic to a disjoint union of real lines because of the free $\R$-action on the moduli space.
Since $\morsedbar_\tau^{-1}(0) \subset \morsedbar_\tau^{-1}(N)$ also admits the free $\R$-action, $\morsedbar_\tau^{-1}(0)$ is also a disjoint union of real lines. 
Hence, in a neighborhood of $\morsedbar_\tau^{-1}(0)$, we have $\morsedbar_\tau^{-1}(0) = \morsedbar_\tau^{-1}(N)$, and by Lemma~\ref{lemma:clean-intersection-criterion}, $\morsedbar_\tau^{-1}(0)$ is cleanly cut out.
\end{proof}

\begin{example}
Suppose $X$ is a smooth manifold with boundary, and $f: X \to \R$ is a Morse function with some critical points on $\partial X$. The Morse homology of $X$ using $f$ is studied in \cite{kronheimer2007Monopoles}. In this case, there does not always exist a Riemannian metric $g$ on $X$ such that the pair $(f,g)$ is Morse-Smale. Instead, it is shown in \cite{bao2024involutive} that a generic metric makes all moduli spaces cleanly cut out. 
\end{example}

A generalization of this example is a manifold with a finite group action. 
\begin{question}
    For any fixed equivariant Morse function, does a generic equivariant metric make all the moduli spaces cleanly cut out?
\end{question}

\begin{example}
Let $(M,\omega)$ be a closed symplectic manifold of dimension $\geq 6$. Theorem A of \cite{wendl2023transversality} implies that for a generic compatible almost complex structure, pseudo-holomorphic curves of index $0$ are super-rigid. Given a super-rigid curve, its $d$-fold branched covers are cleanly cut out.
\end{example}

\begin{lemma}
    Suppose that $\parametrizedModF{p}{q}$ is cleanly cut out. Then $\parametrizedModF{p}{q}$ is a disjoint union of smooth manifolds. The dimension of each connected component is given by $\dim \ker D_u$ for any $u$ in that component. In particular, the dimension of $\ker D_u$ (and hence the dimension of $\coker D_u$) is locally constant over $\parametrizedModF{p}{q}$.
\end{lemma}

\begin{proof}
    This follows directly from Proposition~\ref{prop:local-coordinate}.
\end{proof}

Let $D_u^*$ be the adjoint operator of $D_u$ with respect to the $L^2$-inner product induced by the metric. We have $\coker D_u$ is isomorphic to $\ker D_u^*$. We define the obstruction bundle $\obstructionBundle(p,q) \to \parametrizedModF{p}{q}$ by $\obstructionBundle(p,q)_u = \ker D_u^*$. Then the section $\morsedbar: \widetilde{\mathcal B}(p,q) \to \mathcal E(p,q)$ is transverse to the sub-bundle $\mathcal O(p,q) \subset \mathcal E(p,q)$, so $\morsedbar^{-1}(\mathcal O(p,q)) \subset \widetilde{\mathcal B}(p,q)$ is a smooth submanifold of dimension $ \ind(p) - \ind(q) + \dim \obstructionBundle(p,q)_u = \dim \ker D_u$. Since $\parametrizedModF{p}{q} = \morsedbar^{-1}(0) \subset \morsedbar^{-1}(\obstructionBundle(p,q))$, we have $\parametrizedModF{p}{q} = \morsedbar^{-1}(\obstructionBundle(p,q))$.
The $\R$-action on $\parametrizedModF{p}{q}$ induces an identification of the fibers of $\obstructionBundle$, so we get a vector bundle over $\modF{p}{q}$, still denoted by $\obstructionBundle \to \modF{p}{q}$.

\section{Gluing} \label{section: gluing}
Let $p, q, r \in \crit(f)$, $\M_- = \modF{p}{q}/\R$, and $\M_+ = \modF{q}{r}/\R$. We glue the moduli spaces $\M_+$ with $\M_-$.

\begin{figure}
    \centering
    \begin{tikzpicture}[scale = 0.3]
        \draw[arrows = {-Latex[width=5pt, length=10pt]}](0,4) ..controls (1.5,2) .. (2,0);
        \draw[arrows = {-Latex[width=5pt, length=10pt]}] (4,8)..controls (1.5,6.5) ..  (0,4);
        \draw[->](4,4) -- (7,4);
        \draw[arrows = {-Latex[width=5pt, length=10pt]}](8,0) -- (8,8);
        \node[] at (9, 4){$\R$};
        \node[] at (5.5, 5){$f$};
        \node[anchor = east] at (1, 2){$u_+$};
        \node[anchor = east] at (2, 7){$u_-$};
        \filldraw[black] (2,0)  circle (1pt) node[anchor = east]{$r$};   
        \filldraw[black] (0,4)  circle (1pt) node[anchor = east]{$q$};
        \filldraw[black] (4,8)  circle (1pt) node[anchor = south]{$p$};
    \end{tikzpicture}
    \caption{A broken trajectory}
    \label{fig:enter-label}
\end{figure}
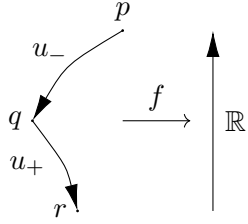

\subsubsection{Choice of representative} \label{subsubsection: choice of representative}
For a small constant $\varepsilon > 0$, we choose the unique representatives $u_\pm$ of $[u_\pm] \in \M_\pm$ as follows: there exist $l_+, l_- > 0$ such that 
\begin{enumerate}
    \item $u_+(-l_+) = f(q) - \varepsilon$ and $u_+(l_+) = f(r) + \varepsilon$
    \item $u_-(-l_-) = f(p) - \varepsilon$ and $u_-(l_-) = f(q) + \varepsilon$
\end{enumerate}
This is possible as long as $\varepsilon$ is smaller than the gaps between the values of $f$ at different critical points.

\subsection{Construction of obstruction bundle gluing} \label{section: pregluing}
In this section, we review the obstruction bundle gluing originally defined by Hutchings-Taubes.
Let $T$ be a large gluing parameter and $r$ be a fixed large constant. We define the pregluing of $[u_+] \in \M_+$ with $[u_-] \in \M_-$. We first choose representatives $u_\pm$ as explained in Section~\ref{subsubsection: choice of representative}.

Choose a cutoff function $\beta: \R \to [0,1]$ such that $\beta(s)=0$ for $s\leq -1$ and $\beta(s)=1$ for $s\geq 1$. 
Let $\bMinus{T}(s) = \beta(\frac{T-s}{r})$ and $\bPlus{T}(s) = \beta(\frac{T+s}{r})$.
See Figure~\ref{fig:branch jumps at node}.

\begin{figure} 
\centering
\def\h{0.5}
\def\hh{0.5}
\def\ss{0.3}
\def\hhh{0.6}

\begin{tikzpicture}[scale = 0.8, every node/.style={scale=0.6}]

\draw[->] (-7, 0) -- (7,0);
\filldraw[black] (-4,0)  circle (1pt) node[anchor = north]{$-T -h$};
\filldraw[black] (-3,0)  circle (1pt) node[anchor = north]{$-T + h$};
\filldraw[black] (0,0)  circle (1pt) node[anchor = north]{$0$};
\filldraw[black] (3,0)  circle (1pt) node[anchor = north]{$T-h$};
\filldraw[black] (4,0)  circle (1pt) node[anchor = north]{$T + h$};

\draw[red] (-5,\h) -- (-4, \h) .. controls (-4 + \ss, \h) and (-3 - \ss, \h + \hh) .. (-3, \h + \hh) -- (5, \h + \hh);
\draw[blue] (-5, \hhh + \hh ) -- (3, \hhh + \hh) .. controls (3 + \ss,  \hhh + \hh) and (4 - \ss, \hhh ).. (4, \hhh) -- (5, \hhh);

\node at (-5.2, \h){$\bPlus{T}$};
\node at (5.2, \hhh){$\bMinus{T}$};

\def \rr{0}
\def \hhhh{-0.5}
\draw[brown] (-2 + \rr , 1+ \hhhh) .. controls (-1+ \rr, 1+ \hhhh) ..  (2+ \rr, 3+ \hhhh);
\node at (2+ \rr + 0.3, 3+ \hhhh){$u_+$};

\def \rr{4.5}
\def \hhhh{-0.5}
\draw[brown] (-2 + \rr , 1+ \hhhh) .. controls (-1+ \rr, 1+ \hhhh) ..  (2+ \rr, 3+ \hhhh);
\node at (2+ \rr + 0.5, 3+ \hhhh){$\uPlus{-2T}$};

\def \rr{0}
\def \hhhh{-0.5}
\draw[cyan] (2 - \rr , 1+ \hhhh) .. controls (1- \rr, 1+ \hhhh) ..  (-2- \rr, 3+ \hhhh);
\node at (-2- \rr - 0.3, 3+ \hhhh){$u_-$};

\def \rr{4.5}
\def \hhhh{-0.5}
\draw[cyan] (2 - \rr , 1+ \hhhh) .. controls (1- \rr, 1+ \hhhh) ..  (-2- \rr, 3+ \hhhh);
\node at (-2- \rr + 0.5, 3+ \hhhh){$\uMinus{2T}$};

\end{tikzpicture}
\caption[]{Cutoff functions and translated maps.}
\label{fig:branch jumps at node}
\end{figure}
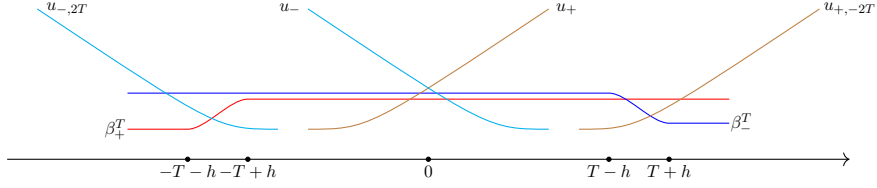

Let $\trans{T}: \R \to \R$ be translation by $T$, i.e., $\trans{T}(s)= s + T$,
and $\uPlus{T} = u_+ \circ \trans {T}$.
Note that $\trans{T}^* (u_+ ^* TM) = \uPlus{T}^* TM.$
Let $\Perturb_+ \in L_1^2(u_+^*TM)$,
and let $\pPlus{T} = \trans{T}^* \Perturb \in L_1^2(\uPlus{T}^*TM)$.
Similarly, we can define $\uMinus{2T}$ and $\pMinus{2T}$.

We identify a neighborhood of $q$ inside $M$ with a ball centered at the origin in $\R^n$.
We choose a metric $g_0$ on $M$ that agrees with $g$ outside the neighborhood of $q$ and equals the standard Euclidean metric inside a smaller neighborhood of $q$. We use this metric $g_0$ to define exponential maps: $u_\pm^* N_\pm \to B(u_\pm)$, where $N_\pm$ is the normal bundle to $u_\pm$ in $M$, and $B(u_\pm)$ is a tubular neighborhood of $u_\pm.$
With this choice, the exponential maps near $q$ are given by addition in $\R^n$.
Since all the interpolations for gluing happen near $q$, we do not distinguish exponential maps from addition in notation. Note that for the gradient vector field, we still use the original metric $g$.

We define the pregluing as follows:
\begin{equation}\label{eqn: pregluing}
u = \bPlus{T}(\uPlus{-2T} + \pPlus{-2T}) + \bMinus{-T}(\uMinus{2T} + \pMinus{2T}).
\end{equation}
Let $v := \bPlus{T}\uPlus{-2T} + \bMinus{-T}\uMinus{2T}$ and $\Phi: u^*TM \to v^*TM$ be the identification given by parallel transportation along the path $\{u_\tau\}_{0\leq \tau \leq 1}$, where $u_\tau$ is defined in the same way as $u$ but with $\psi_\pm$ replaced by $\tau \psi_\pm$, using the metric $g_0$.
Similarly, we identify $(\op{Exp}_{u_{-, 2T}} \psi_{-, 2T})^* TM$ with $u_{-, 2T}^* TM$, and  $(\op{Exp}_{u_{+, -2T}} \psi_{+, -2T})^* TM$ with $u_{+, -2T}^* TM$.
Then the equation $\morsedbar u = 0$ or more precisely $\Phi \morsedbar u = 0$ can be expressed as:
\begin{align}\label{eqn: Lu}
\Phi\morsedbar u &= \bPlus{T} \left(D_{\uPlus{-2T}} \pPlus{-2T} + \frac{d}{ds}(\bMinus{-T})(\uMinus{2T} + \pMinus{2T})\right) \nonumber \\
&\quad+ \bMinus{-T} \left(D_{\uMinus{2T}}\pMinus{2T} + \frac{d}{ds}(\bPlus{T})(\uPlus{-2T} + \pPlus{-2T})\right) = 0. 
\end{align} 
It suffices for the following equations to be satisfied:
\begin{equation} \label{eqn: Theta plus 2T} 
\Theta_+^{-2T}: = D_{\uPlus{-2T}} \pPlus{-2T} + \frac{d}{ds}(\bMinus{-T})(\uMinus{2T} + \pMinus{2T}) = 0,
\end{equation}
\begin{equation}\label{eqn: Theta minus 2T}
\Theta_-^{2T}: = D_{\uMinus{2T}}\pMinus{2T} + \frac{d}{ds}(\bPlus{T})(\uPlus{-2T} + \pPlus{-2T}) = 0.
\end{equation} 

These equations become the following equations after applying $\trans{\pm 2T}^*$:
\begin{equation} \label{eqn: Theta plus} 
\Theta_+ := D_{u_+} \psi_+ + \frac{d}{ds}(\bMinus{T})(\uMinus{4T} + \pMinus{4T}) = 0
\end{equation}
\begin{equation} \label{eqn: Theta minus}
\Theta_- := D_{u_-} \psi_- + \frac{d}{ds}(\bPlus{-T})(\uPlus{-4T} + \pPlus{-4T}) = 0. 
\end{equation}

For a fixed $(u_+, u_-)$, the system of equations $\Theta_+(\psi_+, \psi_-) = 0$ and $\Theta_-(\psi_+, \psi_-) = 0$ does not always have solutions. We consider the following two equations:
\begin{equation}\label{eqn: Theta hat plus}
\Theta_+^\perp = D_{u_+} \psi_+ + (1 - \Pi_+) \left(\frac{d}{ds}(\bMinus{T})(\uMinus{4T} + \pMinus{4T})\right) = 0
\end{equation}
\begin{equation}\label{eqn: Theta hat minus}
\Theta_-^\perp = D_{u_-} \psi_- + (1 - \Pi_-) \left(\frac{d}{ds}(\bPlus{-T})(\uMinus{4T} + \pMinus{4T})\right) = 0,
\end{equation}
where $\Pi_\pm$ is the orthogonal projection onto $\ker D^*_{u_\pm}$.

\begin{proposition}\label{prop: contraction mapping}
For any $(u_+, u_-)$, there exists $R \in \R$ such that for any $T \geq R$, the equations $\Theta_+^\perp (\psi_+, \psi_-) = 0$ and $\Theta_-^\perp (\psi_+, \psi_-) = 0$ have a unique solution.
\end{proposition}

\begin{proof}[Sketch of proof]
This essentially follows from \cite{hutchings2009gluing}. Let $\mathcal H^\pm = (\ker D_{u_\pm})^\perp \subset L_1^2(u_\pm^*TM)$. Let $D_\pm = D_{u_\pm}|_{\mathcal H^\pm}: \mathcal H^\pm \to \op{im}D_{u_\pm} \subset L^2(u_\pm^*TM)$. Then $D_\pm$ is an isomorphism. We consider the map $\Gamma: \mathcal H^+ \times \mathcal H^- \to \mathcal H^+ \times \mathcal H^-$ defined by
\begin{align*}
(\psi_+, \psi_-) &\mapsto \left(-D_+^{-1}(1 - \Pi_+) \left(\frac{d}{ds}(\bMinus{T})(\uMinus{4T} + \pMinus{4T})\right), \right. \\
&\quad \left.-D_-^{-1}(1 - \Pi_-) \left(\frac{d}{ds}(\bPlus{-T})(\uPlus{-4T} + \pPlus{-4T})\right) \right).
\end{align*}
A standard estimate, as in Proposition 4.6 of \cite{hutchings2009gluing} or Lemma 8.4.4 of \cite{bao2018definition}, shows that $\Gamma$ is a contraction map when restricted to a small ball in $\mathcal H^+ \times \mathcal H^-$. Hence, by the contraction mapping theorem, we have a unique solution.
\end{proof}

\subsection{Obstruction Section}
We define the obstruction bundle $\obstructionBundlePM \to \M_+ \times \M_- \times [R,\infty)$ to be $\op{pr}_+^*\obstructionBundle_+ \oplus \op{pr}_-^*\obstructionBundle_-$, or simply, $\obstructionBundle_+ \oplus \obstructionBundle_-$, where $\op{pr}_\pm: \M_+ \times \M_- \times [R,\infty) \to \M_\pm$ is the projection map, and $\obstructionBundle_\pm \to \M_\pm$ is the obstruction bundle defined by $\obstructionBundle_\pm |_{[u_\pm]} = \ker D_{u_\pm}^*$.

Let $\compactSubset_\pm \subset \M_\pm$ be any two compact subsets, and $\openSubset_\pm$ be open subsets such that $\compactSubset_\pm \subset \openSubset_\pm \subset \M_\pm$. We define the \emph{obstruction section} $\obstructionSectionS_\pm: \openSubset_+ \times \openSubset_- \times [R,\infty) \to \obstructionBundle_\pm$ by
\begin{equation}
\obstructionSectionS_\pm([u_+],[u_-], T) = \Pi_\pm \left((\frac{d}{ds}\bump{\mp}{\pm T})(\umap{\mp}{\pm 4T} + \perturb{\mp}{\pm 4T}) \right),
\end{equation} 
where $\Pi_\pm$ is the orthogonal projection onto $\ker D^*_{u_\pm}$, and $(\psi_+,\psi_-)$ is the unique solution to Equation~\eqref{eqn: Theta hat plus} and \eqref{eqn: Theta hat minus}. 
Then $\obstructionSectionS_\pm = 0$, together with Equations~\eqref{eqn: Theta hat plus} and \eqref{eqn: Theta hat minus}, imply that Equations~\eqref{eqn: Theta plus} and \eqref{eqn: Theta minus}, and hence $\morsedbar u = 0$.

Let $\obstructionSectionS = (\obstructionSectionS_+, \obstructionSectionS_-): \openSubset_+ \times \openSubset_- \times [R,\infty) \to \obstructionBundlePM$, and $\obstructionSectionS^{-1}(0)$ is the space of broken curves and gluing parameters that can be glued.
The dependence of the obstruction section on $([u_+], [u_-], T)$ is implicit through $\psi_{\pm}$. Its main part, \emph{the linearized obstruction section}, defined by $\obstructionSectionS_0 = (\obstructionSectionS_{0,+},\obstructionSectionS_{0,-})$, where
\begin{equation}
\obstructionSectionS_{0,\pm} = \Pi_\pm \left((\frac{d}{ds} \bump{\mp}{\pm T})\umap{\mp}{\pm 4T}\right)
\end{equation}
is computed explicitly later in Equations~\eqref{eqn: s0+} and \eqref{eqn: s0-}.

\subsection{Evaluation maps and linearized sections} \label{section: evaluation maps}
In this section, we define evaluation maps and provide an explicit expression for the linearized section.
In this section make the following additional assumption about the metric $g$ near the critical points, which is {\em not} required in the rest of the paper:
\begin{assumption}
    We assume $(f,g)$ is normal around all the critical points of $f$.
\end{assumption}
\begin{definition}[Normal]
    Let $p$ be a critical point of $f$.
    We say the pair $(f,g)$ is {\em normal around $p$}, if there exists a coordinate chart $x = (x_i)_{i\in \mathcal I}$ of $M$ around $p$ with the index set $\mathcal I = \mathcal I_- \sqcup \mathcal I_+$, where  
    \[
    \mathcal I_-= \{-\ind p, \dots, {-1}\}\hspace{0.5cm} \text{ and }\hspace{0.5cm} \mathcal I_+ = \{ 1, \dots, {n-\ind p}\},
    \]
    such that $$f(x) = f(p)+ \frac{1}{2}\sum_{i \in \mathcal I} \lambda_i x_i^2$$ for some constants $\lambda_i$, satisfying 
    \begin{equation} \label{eqn: lambdai}
        \lambda_{-\ind p} \leq \dots \leq \lambda_{-1} < 0 < \lambda_1 \leq \dots \leq \lambda_{n-\ind p},
    \end{equation}
    and $$g = \sum_{i\in \mathcal I} dx_i^2.$$
\end{definition}
If $(f,g)$ is normal around a critical point $p$, then with respect to the coordinate $x$, the $\grad{f}{g}(x) = \sum_{i\in \mathcal I} \lambda_i x_i \frac{\partial}{\partial x_i}$ is linear in $x$.

Let $\tildeM = \parametrizedModF{p}{q}$ and we define evaluation maps $\widetilde{\op{ev}}_-: \tildeM \to \R^{\ind p}$ by 
$$u \mapsto (c_{-\ind p}, \dots, c_{-1}),$$
where
\begin{itemize}
    \item near $s = - \infty$, 
    \begin{equation} \label{eqn: fourier expansion for solution}
        u(s) = \sum_{i \in \mathcal I_-} c_i e^{-\lambda_i s} v_i
    \end{equation}
    \item $v_i$ and $\lambda_i$ are eigenvectors and eigenvalues of $\hess{f} (p)$, i.e., $$\hess{f}(p) v_i = \lambda_i v_i,$$ in particular, $\lambda_i$ satisfies Formula~\ref{eqn: lambdai} and $v_i = \frac{\partial}{\partial x_i}$ using the coordinate $x$.
\end{itemize}
Similarly, we can define $\widetilde{\op{ev}}_+: \tildeM \to \R^{n-\ind q}$.

In the general case, without assuming $(f,g)$ is normal, the above expression needs to be replaced by the first term plus implicit higher-order terms.  This implies that the evaluation map is defined as $\widetilde{\op{ev}}_{-,1}: \tildeM \to \R$ by $u \mapsto c_{-1},$
which is sufficient for our example calculation.

For elements in the cokernel, we also have an expansion similar to Formula~\ref{eqn: fourier expansion for solution}. More precisely, for any $u \in \tildeM$, let $D_u$ be the linearized operator $\frac{d}{ds} + \hess{f}$ of $\morsedbar$ along $u$. Let $D_{u}^* = - \nabla_s + \hess{f}: L_1^2(u^* TM) \to L^2(u^* TM)$ be the adjoint operator of $D_u$ with respect to the $L^2$-norm. For any $\eta \in \ker D_{u}^*$, near $s = -\infty$, we have
\begin{equation}\label{eqn: cokernel explicit fourier expansion}
    \eta(s) = \sum_{i \in \mathcal I_+} d_i e^{\lambda_i s} v_i.
\end{equation}
We have a similar expression for $\eta$ near $s = +\infty$.

Now we are in a position to calculate the linearized section $\obstructionSectionS_0$ explicitly. The function $\frac{d}{ds}\bMinus{T}(s)$ is supported in $[-T-1, -T+1]$, and $\uMinus{4T}(s) = u(s+4T)$ lies in a small neighborhood of $q$ for $s \in [-T-1, -T+1]$ and large $T \gg 0$. By Equation~\eqref{eqn: fourier expansion for solution}, we have, for $s \in [-T-1, -T+1]$,
$$\uMinus{4T} = \sum_{i \in \mathcal I_+} c^-_i e^{-\lambda^i(s + 4T)}v^i.$$
Moreover, for any $\eta_+ \in \ker D^*_{u_+}$, and $s \in [-T-1, -T+1]$, we have
$$\eta_+(s) = \sum_{i \in \mathcal I_+} d^+_i e^{\lambda^i s} v^i.$$
This implies
\begin{equation}\label{eqn: s0+}
\langle \mathfrak{s}_{0,+}, \eta_+ \rangle = \sum_{i \in \mathcal I_+} c^-_i d^+_i e^{-4\lambda_i T}.
\end{equation}
Similarly, for $s \in [T-1, T+1]$, we can write
$$\uPlus{-4T} = \sum_{i \in \mathcal I_-} c^+_i e^{-\lambda^i(s - 4T)}v^i.$$
For any $\eta_- \in \ker D^*_{u_-}$, and $s \in [T-1, T+1]$, we have
$$\eta_-(s) = \sum_{i \in \mathcal I_-} d^-_i e^{\lambda^i s} v^i.$$
Then we have
\begin{equation}\label{eqn: s0-}
\langle \mathfrak{s}_{0,-}, \eta_- \rangle = -\sum_{i \in \mathcal I_-} c^+_i d^-_i e^{4\lambda_i T}.
\end{equation}
The gluing of broken trajectories that consist of more than two pieces, and the correspoinding obstruction section is explained in Theorem~\ref{thm: simultaneous gluing}. The corresponding linearized section has a similar expression.

\subsection{Glue two level trajectories}
Let $\delta > 0$ be a small constant. Let $\compactSubset_\pm \subset \M_\pm$ be compact subsets. We define the notion of being \emph{close to breaking} in the spirit of Definition 7.1 in \cite{hutchings2009gluing} as follows:

\begin{definition}[Close to breaking]\label{defn:close_to_breaking}
An element $[u] \in \M$ is \emph{$\delta$-close to breaking into the broken trajectory} $([u_+],[u_-]) \in \compactSubset_+ \times \compactSubset_-$ if there exist representatives $u$, $u_+$, and $u_-$ of $[u]$, $[u_+]$, and $[u_-]$ respectively, and $\sigma > 0$, satisfying the following conditions:
\begin{enumerate}[label=(\alph*)]
    \item $u_+|_{(-\infty, \sigma]}$, $u_-|_{[-\sigma, \infty)}$, and $u|_{[-\sigma, \sigma]}$ are $\delta$-close in the $C^1$-norm to constant maps to $q.$
    \item $u|_{[\sigma, \infty)}$ is $\delta$-close in the $C^1$-norm to $u_+|_{[\sigma, \infty)}$.
    \item $u|_{(-\infty, -\sigma]}$ is $\delta$-close in the $C^1$-norm to $u_-|_{(-\infty, -\sigma]}$.
\end{enumerate}
We denote the space of such $\delta$-close to breaking trajectories by $\closeToBreaking{\delta}{\compactSubset_+}{\compactSubset_-}$.
\end{definition}


Let $\M = \modF{p}{r}$, and let $\obstructionBundle \to \M$ be the obstruction bundle.

\begin{theorem}[Gluing]\label{thm: gluing} 
Suppose that $\M_+$, $\M_-$, and $\M$ are cleanly cut out. For any compact subsets $\compactSubset_\pm \subset \M_\pm$, there exist open neighborhoods $\compactSubset_\pm \subset \openSubset_\pm \subset \M_\pm$, a gluing parameter bound $R > 0$, a close to breaking parameter bound $\delta >0$, and a section $\obstructionSectionS$ also known as the obstruction section, of the bundle $\obstructionBundlePM \to \openSubset_+ \times \openSubset_- \times [R, \infty)$ such that the following holds:
\begin{enumerate}[label=(G\arabic*)]
    \item \label{G1} The obstruction section $\obstructionSectionS$ is a $C^1$-map. It intersects the zero section cleanly, 
    and $\obstructionSectionS \to 0$ in $C^1$-norm as $T\to \infty$.
   We denote $\gluableSet = \obstructionSectionS^{-1}(0)$, and let $i: \gluableSet \to \openSubset_+ \times \openSubset_- \times [R, \infty)$ be the inclusion map.

    \item There exists a $C^1$-gluing map $G: \gluableSet \to \M$ such that:
    \begin{enumerate}[label=(\roman*)]
        \item The map $G$ is a diffeomorphism onto its image.
        \item The image of $G$ is essentially the same as $\mathfrak{s}^{-1}(0)$:
        \be 
            \item[-] For any $R' \geq R$, there exists $\delta' > 0$ such that $\closeToBreaking{\delta'}{\compactSubset_+}{\compactSubset_-} \subset G \circ i^{-1}(\openSubset_+ \times \openSubset_- \times [R', \infty))$.
            \item[-] For any $0 < \delta' < \delta$, there exists $R' \geq R$ such that $G \circ i^{-1}(\compactSubset_+ \times \compactSubset_- \times [R', \infty)) \subset \closeToBreaking{\delta'}{\compactSubset_+}{\compactSubset_-}$.
        \ee 
    \end{enumerate}
    \item There exists a bundle map $j^\sharp: G^* \obstructionBundle \to i^* \obstructionBundlePM$ such that:
    \begin{enumerate}[label=(\roman*)]
        \item $j^\sharp$ is injective.
        \item $i^\sharp \circ j^\sharp$ is transverse to $\obstructionSectionS$ in $\obstructionBundlePM$, where $i^\sharp$ is the induced map from $i$. See the following diagram:
    \end{enumerate}
    \begin{equation}\label{fig: original_gluing_diagram}
    \begin{tikzcd}
    \obstructionBundlePM \arrow[d] & i^* \obstructionBundlePM \arrow{dr} \arrow[l, "i^\sharp"] & & G^* \obstructionBundle \arrow[ll, "j^\sharp"] \arrow{dl} \arrow[r, "G^\sharp"] & \obstructionBundle \arrow[d] \\
    \openSubset_+ \times \openSubset_- \times [R,\infty) \arrow[u, bend left, "\obstructionSectionS"] & & \gluableSet \arrow[ll, "i"] \arrow[rr, "G"] & & \M \\
    \end{tikzcd}.
    \end{equation}
\end{enumerate}
\end{theorem}

\begin{remark}
Let $B = i \circ G^{-1}: \closeToBreaking{\delta}{\compactSubset_+}{\compactSubset_-} \to \openSubset_+ \times \openSubset_- \times [R, \infty)$, and let $B^\sharp = i^\sharp \circ j^\sharp \circ (G^\sharp)^{-1}: \obstructionBundle|_{\closeToBreaking{\delta}{\compactSubset_+}{\compactSubset_-}} \to \obstructionBundlePM$. We refer to $B$ as the \emph{breaking map} in contrast to the gluing map. We can simplify Diagram~\eqref{fig: original_gluing_diagram} as follows:
\begin{equation}\label{fig: simplified gluing diagram}
\begin{tikzcd}
\obstructionBundlePM \arrow[d] & \obstructionBundle \arrow[d] \arrow[l, "B^\sharp"] \\
\openSubset_+ \times \openSubset_- \times [R,\infty) \arrow[u, bend left, "\obstructionSectionS"] & \closeToBreaking{\delta}{\compactSubset_+}{\compactSubset_-} \arrow[l, "B"] \\
\end{tikzcd}
\end{equation}
\end{remark}


Before going into the proof of the Gluing Theorem~\ref{thm: gluing}, we look at its application on perturbing the moduli spaces:

\subsection{A Perturbation Section}\label{section: perturbation section}
In this section, we focus on the easy case when $\M_\pm$ are compact, so $\openSubset_\pm = \M_\pm$. The general case is explained in Section~\ref{section: construction of semi-global}.

Let $\sigma_\pm$ be a section of $\obstructionBundle_\pm \to \M_\pm$, and $T_0 \in [R, \infty)$.
Denote by $\M_{T} = \M \backslash B^{-1}(\M_+ \times \M_- \times (-\infty, T))$ for any $T \in [R, \infty)$.

\begin{definition}[Compatible Section]
    We say $\{\sigma_{+-}, \sigma\}$ is a \emph{compatible section} if:
    \begin{itemize}
        \item $\sigma$ is a section of $\obstructionBundle \to \M_{T_0}$.
        \item $\sigma_{+-}$ is a section of $\obstructionBundlePM \to \M_+ \times \M_- \times [T_0, \infty)$.
    \end{itemize} 
    and they satisfy:
\begin{enumerate}[label = ($p$\arabic*)] 
\item \label{sigma1} (Compatibility with the interior) $\sigma_{+-} \circ B =  B^\sharp \circ \sigma$ over $B^{-1}(\M_+ \times \M_- \times [T_0, T_0 + 1])$. 
\item \label{sigma2} (Compatibility with lower strata) $\sigma_{+-} = (\sigma_+, \sigma_-)$ for $T \gg T_0$ sufficiently large.
\end{enumerate}
We define the perturbed moduli space to be $$\mathcal Z = \sigma^{-1}(0) \sqcup \sigma_{+-}^{-1}(0)/\sim$$
where $\sim$ is given by \ref{sigma3}.

\end{definition}

\begin{lemma}[Perturbation Section]\label{lemma: C1small}
    Assume the same conditions as Theorem~\ref{thm: gluing}. Suppose that $\M_\pm$ are compact, and that $\sigma_\pm$ are transverse to the zero sections. Let  $\{\sigma_{+-}, \sigma\}$ be a compatible section such that:
    \begin{enumerate}[label = ($p$\arabic*), start = 3] 
        \item \label{sigma3} ($C^1$-small) $\sigma_{+-}$ and $\sigma$ are $C^1$-small;
        \item \label{sigma4} (Transverse) $(\mathfrak s + \sigma_{+-})$ and $\sigma$ are transverse to the zero sections, respectively.
    \end{enumerate}
    Then $\mathcal Z$ is a smooth manifold of dimension $\op{virdim} \M = \ind p - \ind q - 1$, and we call $\{\sigma_{+-}, \sigma\}$ a \emph{perturbation section}.
If $\op{virdim} \mathcal M = 1$, then for any generic $T > T_0$, the {\em $T$-boundary} of $\mathcal Z$ defined by 
    $$\partial^T \mathcal Z := (\M_+ \times \M_- \times \{T\}) \cap \mathcal Z$$
is cobordant to
$$\sigma_+^{-1}(0) \times \sigma_-^{-1}(0).$$
\end{lemma}
 \begin{proof}
     By \ref{sigma3}, over $(T_0, T_0 + 1)$ we have $\sigma^{-1}(0) \subset (\obstructionSectionS + \sigma)^{-1}(0)$. 
     Since $\obstructionSectionS$ is transverse to $\obstructionBundle$ and $\sigma_{+-}$ is $C^1$-small,
     over $(T_0, T_0 + 1)$, we have $\sigma^{-1}(0) = (\obstructionSectionS + \sigma)^{-1}(0)$. 
     Hence for any $T, T' > T_0$, we have $\partial^T \mathcal Z$ is cobordant to $\partial^{T'} \mathcal Z$.
     When $T'$ is sufficiently large, by \ref{sigma4} and the fact that $\obstructionSectionS \to 0$ as $T\to \infty$ in $C^1$-norm from \ref{G1} in Theorem~\ref{thm: gluing}, we have $\partial^T \mathcal Z$ is in bijection with $\sigma_+^{-1}(0) \times \sigma_-^{-1}(0)$.
 \end{proof}


\subsection{Construction of a perturbation section} \label{section: construction of a perturbation section}
We now construct a perturbation section. Let $\sigma$ be a transverse section of $\obstructionBundle\to \M$.
First, extend $B^\sharp \circ \sigma$, viewed as a section of $\obstructionBundlePM|_{\op{im}B}$, to a section $\tilde\sigma$ of $\obstructionBundlePM$ such that $\tilde\sigma$ vanishes outside a tubular neighborhood $\mathcal{U}$ of the image of $B$. Pick $T_1 > T_0$. We define $\sigma_{+-} = (\sigma_+, \sigma_-)$ over the region $T \geq T_1$; $\sigma_{+-} = \tilde\sigma$ over the region $T_0 \leq T \leq T_0 + 1$; and we choose $\sigma_{+-}$ to be a generic section over the region $\{ T_0 + 1 \leq T \leq T_1\}$. We now verify the requirements \ref{sigma1}-\ref{sigma4}. Requirements \ref{sigma1}-\ref{sigma3} are straightforward by construction. For \ref{sigma4}:

\begin{enumerate}[label=(\alph*)]
    \item Over the region $T \geq T_1$, since $\sigma_\pm$ are transverse to the zero section and $\|\mathfrak{s}\|_{C^1} \to 0$ as $T \to \infty$, $\sigma_{+-} + \mathfrak{s}$ is transverse to the zero section when $T_1$ is sufficiently large.
    \item Over the region $T_0 \leq T \leq T_0 + 1$, let $x$ denote the coordinate of $\op{im}(B)$, and $y$ denote the coordinate of the normal direction. Then $\tilde\sigma(x, y) = (\beta(y) \sigma(x), 0) \in \obstructionBundle \oplus \obstructionBundle^{\perp} = \obstructionBundlePM$, where $\beta$ is a bump function that equals $1$ around $y = 0$, and $\obstructionBundle^{\perp}$ is a normal subbundle of $\obstructionBundle$ within $\obstructionBundlePM$. The support of $\beta$ is detailed below. Suppose $\mathfrak{s}(x, y) = (a(x, y), b(x, y))$. Then $\mathfrak{s}^{-1}(0) = \{ y = 0\}$ implies $b(x, 0) = 0$ and $a(x, 0) = 0$. Note that $\frac{\partial b}{\partial y}(x, 0)$ is surjective by (G3)(ii). Therefore, if $y$ lies in a sufficiently small neighborhood $W$ of $0$, then $\{b(x, y) = 0\} = \{ y = 0\}$. Thus, $\{\tilde\sigma + \mathfrak{s} = 0\} = \{b(x, y) = 0, ~ \beta(y) \sigma(x) + a(x, y) = 0\} = \{(x, 0) \mid \sigma(x) = 0\}$. If $y$ is outside $W$, $\tilde\sigma = 0$ as $\beta(y)$ can be chosen to have support in $W$. Hence, $\{\tilde\sigma + \mathfrak{s} = 0\} = \{(x, 0) \mid \sigma(x) = 0\}$ holds. Thus, to verify $\obstructionSectionS + \sigma_{+-}$ is transverse to the zero section as in \ref{sigma4}, it suffices to check over $y = 0$, which is guaranteed by the fact that $\sigma$ is transverse to zero in $\obstructionBundle$ and $\frac{\partial b}{\partial y}(x, 0)$ is onto.
\end{enumerate}


\subsection{An example}
\begin{example}\label{example: upright torus}
We compute the Morse homology of the \textit{upright torus}. Let $p$, $q$, and $r$ denote critical points of the height function, as depicted in Figure~\ref{fig:torus}. Notably, $\M_- = \M(p, q)$ is transversely cut out and contains two elements: the front trajectory $u$ and the back trajectory $u'$. These trajectories cancel each other in the definition of the differential $\partial$ (see Section~\ref{section: morse homology}), hence $\langle \partial p, q \rangle = 0$.

The moduli space $\M_+ = \M(q, r)$ consists of two elements: the left trajectory $v$ and the right trajectory $v'$. While $\M_+$ is cleanly cut out, it is not transverse and has an obstruction bundle of rank $1$, denoted as $\obstructionBundle_+$. We choose a section $\sigma_+$ of $\obstructionBundle_+ \to \M_+$ that is transverse to the zero section. This is equivalent to choosing two non-zero numbers: $\sigma_+(v) \in \R - \{0\}$ and $\sigma_+(v') \in \R - \{0\}$, if we trivialize the bundle $\obstructionBundle_+$. Therefore, the perturbed moduli space from $q$ to $r$ is $\sigma_+^{-1}(0) = \emptyset$.

The moduli space $\M = \M(p, r)$ is empty and hence transversely cut out. Despite this, we still need to perturb it within its boundary chart: $\obstructionBundle_+ \to \M_+ \times \M_- \times [R, \infty)$. We choose a transverse section $\sigma_{+-}$ of $\obstructionBundle_+ \to \M_+ \times \M_- \times [R, \infty)$ such that it equals $\mathfrak{s}$ near $T = R$ and matches $\sigma_+$ as $T \to \infty$, i.e., $\lim_{T \to \infty} \sigma_{+-}(u_+, u_-, T) = \sigma_+(u_+)$.

Note that $\M_+ \times \M_- = \{(v, u), (v', u), (v, u'), (v', u')\}$. Referring to Figure~\ref{fig: upright torus obstruction section}, precisely one of the two pairs $\{(v, u), (v, u')\}$ glues to a perturbed trajectory from $p$ to $r$. The same holds for the two pairs $\{(v', u), (v', u')\}$. In total, there are two perturbed trajectories from $p$ to $r$, and they cancel each other out. Let $\partial$ be the boundary operator (see Section~\ref{section: morse homology}). Hence, $\langle \partial p, r \rangle = 0$. Therefore, $\partial p = 0$. In the same way, one can show that $\langle \partial q, s \rangle = 0$, and hence $\partial q = 0$. Finally, it is obvious that we have $\partial r = 0 = \partial s$.

\begin{figure}
    \centering
    \begin{subfigure}[b]{.4\textwidth}
    \centering
    \begin{tikzpicture}[scale = 0.8]
    \draw (0,0) ellipse (2cm and 3cm);
    \draw (-0.1, 1.2) .. controls (0.5,1) and (0.5,-1) .. (-0.1, -1.2);
    \draw (0.1, 1) .. controls (-0.5, 1) and (-0.5, -1) .. (0.1,-1);

    \draw[red, fill] (-0, 3) circle (0.05);
    \node[anchor = south] at (0, 2.9) {$p$};
    \draw[red, fill] (0, 1.05) circle (0.05);
    \node[anchor = south] at (0, 1.1) {$q$};

    \draw[red, fill] (0, -1.05) circle (0.05);
    \node[anchor = north] at (0, -1.1) {$r$};

    \draw[red, fill] (0, -3.) circle (0.05);
    \node[anchor = north] at (0, -3.1) {$s$};
    
    \draw[blue] (0.2, 2.98) .. controls (-0.7, 3.1) and (-0.6, 1).. (0, 1.05);
    \node[] at (-0.7, 2) {$u$};

    \draw[dotted, blue] (0.2, 2.98) .. controls (0.75, 2.3) and (0.55, 0.9).. (0, 1.05);
    \node[] at (0.7, 2) {$u'$};

    \draw[red] (0, 1.05) .. controls (-0.5, 1) and (-0.7, -1).. (0, -1.05);
    \node[] at (-0.7, 0) {$v$};

    \draw[dotted, red] (0, 1.05) .. controls (0.5, 1) and (0.7, -1).. (0, -1.05);
    \node[] at (0.7, 0) {$v'$};
    
    \end{tikzpicture}
    \caption{The upright torus}
    \label{fig:torus}
    \end{subfigure}
    \begin{subfigure}[b]{.4\textwidth}
    \centering
    \begin{tikzpicture}[rotate = 90, scale = 0.9]
    \draw[smooth] (0,1) to[out=30,in=150] (2,1) to[out=-30,in=210] (3,1) to[out=30,in=150] (5,1) to[out=-30,in=30] (5,-1) to[out=210,in=-30] (3,-1) to[out=150,in=30] (2,-1) to[out=210,in=-30] (0,-1) to[out=150,in=-150] (0,1);
    \draw[smooth] (0.4,0.1) .. controls (0.8,-0.25) and (1.2,-0.25) .. (1.6,0.1);
    \draw[smooth] (0.5,0) .. controls (0.8,0.2) and (1.2,0.2) .. (1.5,0);
    \draw[smooth] (3.4,0.1) .. controls (3.8,-0.25) and (4.2,-0.25) .. (4.6,0.1);
    \draw[smooth] (3.5,0) .. controls (3.8,0.2) and (4.2,0.2) .. (4.5,0);
    \end{tikzpicture}
    \caption{The upright genus two surface}
    \label{fig:enter-label}
    \end{subfigure}
    \caption{Example~\ref{example: upright torus} and Example~\ref{example: upright genus two}}
\end{figure}
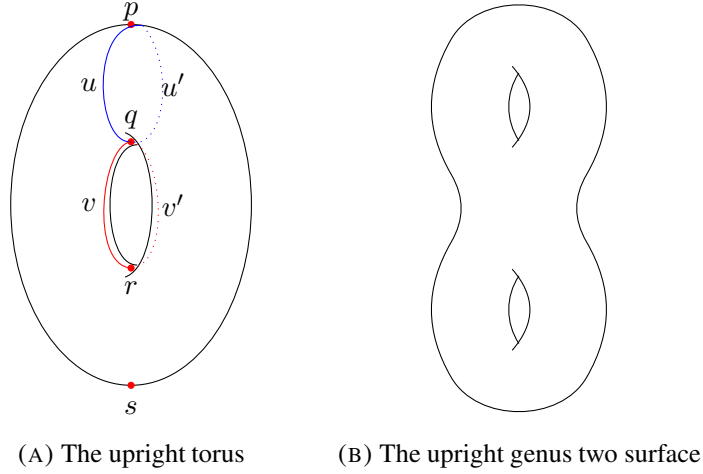

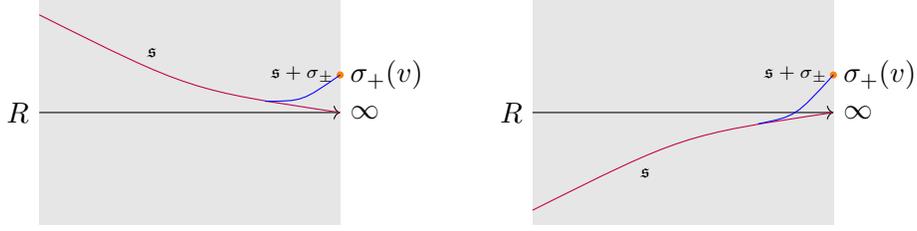
\begin{figure}
    \centering
    \begin{subfigure}[t]{.45\textwidth}
    \begin{tikzpicture}[scale = 1]
    \fill [black!10] (0,0) rectangle (4,3);
    \draw[->] (0,1.5) -- (4,1.5);
    \node[anchor = east] at (0, 1.5) {$R$};
    \node[anchor = west] at (4, 1.5) {$\infty$};
    \node[anchor = west] at (4, 2) {$\sigma_+(v)$};
    \draw [fill, orange] (4, 2) circle (0.04);
    \draw[purple] (0,2.8) .. controls (2,1.8) .. (4, 1.5);
    \node[] at (1.5, 2.3) {\tiny $\mathfrak{s}$};
    \node[] at (3.5, 2){\tiny $\obstructionSectionS + \sigma_\pm$};
    \draw[blue] (3,1.65) .. controls (3.5, 1.65) .. (4, 2);
    \end{tikzpicture}
    \caption{The obstruction bundle restricted to ${\{(v,u)\}\times [R,\infty)}$. 
    After perturbation, $\obstructionSectionS + \sigma_{+-}$ has no zero.}
    \label{fig:enter-label}
    \end{subfigure} \quad \quad
    \begin{subfigure}[t]{.45\textwidth}
    \begin{tikzpicture}[scale = 1]
    \fill [black!10] (0,0) rectangle (4,3);
    \draw[->] (0,1.5) -- (4,1.5);
    \node[anchor = east] at (0, 1.5) {$R$};
    \node[anchor = west] at (4, 1.5) {$\infty$};
    \node[anchor = west] at (4, 2) {$\sigma_+(v)$};
    \draw [fill, orange] (4, 2) circle (0.04);
    \draw[purple] (0,-2.8+3) .. controls (2,-1.8+3) .. (4, -1.5+3);
    \node[] at (1.5, -2.3+3) {\tiny $\obstructionSectionS$};
    \node[] at (3.5, 2){\tiny $\obstructionSectionS + \sigma_\pm$};
    \draw[blue] (3, 1.65-0.3) .. controls (3.5, 1.65-0.2) .. (4, 2);
    \end{tikzpicture}
    \caption{The obstruction bundle restricted to ${\{(v, u')\}\times [R,\infty)}$. 
    After perturbation, $\obstructionSectionS + \sigma_{+-}$ has one zero.}
    \label{fig:enter-label}
    \end{subfigure}
    \caption{The obstruction section $\mathfrak s(v, u,T) = - \mathfrak s(v, u',T)$, as the trajectories $u$ and $u'$ point in opposite directions near $q$. We choose the perturbation section $\sigma_{+-}$ such that $\sigma_{+-} = 0$ when $T$ is near $R$, and $\sigma_{+-} = \sigma_+$ when $T$ is sufficiently large.}
    \label{fig: upright torus obstruction section}
\end{figure}

\end{example}

\begin{remark}
The upright torus can be seen as a special moment of $1$-parameter family of deformations of the torus, as described in Lemma~\ref{lemma: generic family}. The four possible choices of $(\sigma_+(v), \sigma_+(v'))$ correspond to four ways of tilting the torus: tilting it forward, tilting it backward, bending the top and the bottom forward, bending the top and bottom backward.
\end{remark}

\begin{example}[Upright higher genus Riemann surfaces]\label{example: upright genus two}
    For a genus $2$ Riemann surface, the three middle moduli spaces connecting adjacent critical points of index $1$ are obstructed. We leave this example for motivated readers. On the other hand, a simpler calculation using obstruction bundle gluing, but without abstract perturbation, is available in \cite{bao2024involutive}, in the spirit of perturbing the Morse function.
\end{example}

\section{Linear gluing} 
To prove the gluing Theorem~\ref{thm: gluing}, we need the linear version of gluing. This section is a straightforward modification of Section 9.2 in \cite{hutchings2009gluing} for the Morse homology setting.

\begin{definition}
For each pair of  critical points $p, q \in \crit{f}$, a \emph{linear tuple} is a tuple $\mathfrak{O}(p,q) = (E, a, b, g, \nabla, \bs\psi)$, where:

\begin{enumerate}
    \item $E$ is a rank $n$ vector bundle over $\R$;
    \item $a$ and $b$ are real numbers with $a < b$;
    \item The intervals $(-\infty, a] \subset \R$ and $[b, +\infty) \subset \R$ are referred to as the negative and positive ends, respectively. These intervals are identified with $(-\infty, 0]$ and $[0, \infty)$, respectively, via translations;
    \item $\bs\psi = (\psi_q, \psi_p)$ consists of bundle isomorphisms:
    \begin{enumerate}
        \item $\psi_q: E|_{[b, \infty)} \simeq T_q M \times [0, \infty)$, and
        \item $\psi_p: E|_{(-\infty, a]} \simeq T_p M \times (-\infty, 0]$,
    \end{enumerate}
    where the base maps are the aforementioned identifications;
    \item $g$ is a bundle metric that agrees with the metric on $T_q M$ and $T_p M$ at the ends via the trivialization $\bs\psi$;
    \item $\nabla$ is a connection on the bundle $E$ that is compatible with $g$.
\end{enumerate}
\end{definition}

\begin{definition}
    Given $\mathfrak{O}(p,q)$, we define $\mathcal D({\mathfrak{O}(p,q)})$ to be the set of differential operators $D: C^\infty (E) \to C^\infty (E)$ such that for any $\xi \in C^\infty(E)$,
    \[
    (D \xi)(s) = \nabla_s \xi(s) + H(s)\xi(s),
    \]
    where $H(s)\in \op{End}(E_s)$ is a linear map satisfying $\lim_{s\to \infty} H(s) = \hess{f}(q)$ and $\lim_{s \to -\infty} H(s) = \hess{f}(p).$ 
\end{definition}
The following is a standard result. See, for example, \cite{schwarz1995cohomology}.
\begin{lemma}
    The map $D$ extends to a Fredholm operator $L_1^2(E) \to L^2(E)$.
\end{lemma}

Given linear tuples $\mathfrak{O}_+ = \mathfrak{O}(p,q) = (E_+, a_+, b_+, g_+, \nabla_+, \bs\psi_+)$ and $\mathfrak{O}_- = \mathfrak{O}(q,r) = (E_-, a_-, b_-, g_-, \nabla_-, \bs\psi_-)$, and a large number $R > 0$, we can glue them to obtain a linear tuple $\mathfrak{O}$ as follows:
We glue the subset $(-\infty, b_- + 2R]$ of the base of $E_-$ and the subset $[-2R + a_+, \infty)$ of the base of $E_+$ by identifying them at the boundary.
We define $E$ such that $E|_{(-\infty, 0]} = E_-|_{(-\infty, b_- + 2R]}$ and $E|_{[0,\infty)} = E_+|_{[-2R + a_+, \infty)}$, with obvious translations that identify the domains, and at $0$ they are identified via $\bs\psi_+$ and $\bs \psi_-$. The metric $g$ and the connection $\nabla$ are constructed in a straightforward way.

For any $D_\pm \in \mathcal D(\mathfrak{O}_\pm)$, we can glue them with some choices to obtain  an operator $D\in \mathcal D(\mathfrak{O})$. Specifically, we consider an operator $D\in \mathcal D(\mathfrak{O})$ such that over $(-\infty, -R]$, $D$ agrees with $D_-$ (after the above translation), and over $[R, \infty)$, $D$ agrees with $D_+$ (after the above translation). 
Over $[-2R, 2R]$, $E$ is identified via $(\bs\psi_+, \bs\psi_-)$ with $T_q M \times [-2R, 2R]$, so we can define the zeroth-order operators $D-D_\pm$. 

Let $V_\pm$ be finite-dimensional subspaces of $L^2(E)$, $W_\pm$ be the orthogonal complements of $V_\pm$, and $\Pi_{W_\pm}: L^2(E) \to W_\pm$ be the orthogonal projection.  



\begin{lemma}[Linear Gluing -- Proposition 9.3 in \cite{hutchings2009gluing}]
\label{lemma: linear gluing}
Suppose that $\Pi_{W_\pm} D_\pm : L^2(E) \to W_\pm$ are surjective. If the gluing parameter $T$ is sufficiently large, and if over the region $[-2R, 2R]$, $|D-D_+|$ and $|D-D_-|$ are sufficiently small with respect to the operator norm, then we have the short exact sequence:
    $$0 \to \ker D \xrightarrow{} \ker (\Pi_{W_+}  \dPlus) \oplus \ker (\Pi_{W_-}  \dMinus) \xrightarrow{} V_+ \oplus V_- \xrightarrow{h} \coker D \to 0.$$
\end{lemma} 
\begin{proof}
    This was proven for the pseudoholomorphic curve case in \cite{hutchings2009gluing}. The proof for the Morse trajectory case is essentially the same but simpler, so we omit it. For our application, we explain the map $h: V_+ \oplus V_- \xrightarrow{} \coker D$. 
    Let $\beta: \R \to [0,1]$ be a cutoff function such that $\beta(s) = 0$ for $s \geq 1$ and $\beta(s) = 0$ for $s\leq 1$. Let $\beta_+(s) = \beta(-s/R)$ and $\beta_-(s) = \beta(s/R)$.
    For any $v_\pm \in V_\pm$, $h(v_+,v_-)$ is defined to be the equivalence class of $\beta_+ v_+ + \beta_- v_-$ in the $\coker D$.
\end{proof} 

\begin{remark}
If we take $V_\pm$ to be the orthogonal complement of $\op{im}D_\pm$, then $W_\pm = \op{im}D_\pm$, and hence $\Pi_{W_\pm}D_\pm$ is surjective, and $\ker (\Pi_{W_\pm}\circ D_\pm) = \ker D_\pm$.

\end{remark}
\begin{remark}\label{remark: linear gluing}
Consider the dual of the exact sequence $V_+ \oplus V_- \xrightarrow{h} \coker D \to 0$, identify $\coker D$ with $\ker D^*$, and use the inner product to identify a subspace of $L^2(E)$ with its dual. We obtain $V_+ \oplus V_- \xleftarrow{h^*} \ker D^* \leftarrow 0$, where $h^*$ maps $\mathfrak o \in \ker D^*$ to $(\Pi_{V_+}(\beta_+ \mathfrak o) \in V_+ \oplus V_-, \Pi_{V_-}(\beta_- \mathfrak o))$, with $\Pi_{V_\pm}: L^2(E_\pm) \to V_\pm$ being the orthogonal projection.

We consider the subspace of $L^2(E)$ defined by 
\[
    \widetilde V = \{\beta_+ v_+ + \beta_- v_- ~|~ v_\pm \in V_\pm \}.
\]
Let $\phi: V_+ \oplus V_- \to \widetilde V$ be the map sending $(v_+, v_-) \mapsto \beta_+ v_+ + \beta_- v_-$. Then $\phi$ is an isomorphism (if $R$ is sufficiently large).
The map $h$ is the composition of the map $\phi$ and the orthogonal projection $\widetilde V \to \ker D^*$.

The map $h^*$ is given by the composition of the orthogonal projection map $\Pi_{\widetilde V}: \ker D^* \to \widetilde V$ and the map $\widetilde V \to V_+ \oplus V_-$ sending $\widetilde v \mapsto (\Pi_{V_+}\beta_+ \widetilde v, \Pi_{V_-}\beta_- \widetilde v)$.
Since $h^*$ is injective, the projection map $\Pi := \Pi_{\widetilde V}$ is also injective.

Moreover, let $V = \Pi(\ker D^*)$, and let $W$ be the orthogonal complement of $V$ inside $L^2(E)$. 
\begin{claim}\label{claim: projection surjective}
    $\Pi_W D: L^2(E) \to W$ is surjective.
    \end{claim}
    \begin{proof}[Proof of Claim~\ref{claim: projection surjective}]
        Suppose not. Then there exists $0 \neq w \in W$ that is orthogonal to $\op{im}(D)$. This implies $w \in \ker D^*$.
        Therefore, $\Pi(w) \in V$. However, since $V$ is orthogonal to $W$, we have $\Pi(w)$ orthogonal to $w$. This implies $\Pi(w) = 0$, contradicting the fact that $\Pi$ is injective.
    \end{proof}
    
    Furthermore, for any $\lambda \in [0,1]$, consider the operator $\alpha_\lambda: \ker D^* \to L^2(E)$ defined by $\mathfrak o \mapsto \lambda \Pi (\mathfrak o) + (1-\lambda) \mathfrak o$. 
    Consider the subspace of $L^2(E)$ defined by $V_\lambda = \alpha_\lambda(\ker D^*)$.
    Let $W_\lambda$ be the orthogonal complement of $V_\lambda$. Then we have the following claim:
    
    \begin{claim}\label{claim: lambda surjective}
        $\Pi_{W_\lambda} D: L^2(E) \to W_\lambda$ is surjective. 
    \end{claim}
    \begin{proof}[Proof of Claim~\ref{claim: lambda surjective}]
        Suppose not. Then there exists $0 \neq w\in W_\lambda$ that is orthogonal to $\op{im}(D)$, implying $w \in \ker D^*$. Thus, $\alpha_\lambda(w) \in V_\lambda$. Since $V_\lambda$ is orthogonal to $W_\lambda$, we have $\alpha_\lambda(w)$ orthogonal to $w$. 
        Let $e_1, e_2, \dots, e_k$ be an orthonormal basis of $\widetilde V$. Then:
        \begin{align*}
            0 &= \langle w, \alpha_\lambda(w) \rangle \\
            & = \langle w, \sum_{i=1}^k \lambda \langle w, e_i \rangle e_i + (1-\lambda) w \rangle \\
            & = \lambda \sum_{i=1}^k \langle w, \langle w, e_i \rangle e_i \rangle + (1-\lambda) \langle w, w \rangle \\
            & = \lambda \sum_{i=1}^k \langle w, e_i \rangle^2 + (1-\lambda) \langle w, w \rangle.
        \end{align*}
        This implies $\lambda = 1$ and $\Pi(w) = 0$, which contradicts the fact that $\Pi$ is injective.
    \end{proof}
    
    \begin{claim}
        $\ker (\Pi_{W_\lambda} D) = \ker D$.
    \end{claim}
    \begin{proof}
        We show that the restricted map $\Pi_{W_\lambda}|_{\op{im} D}: \op{im} D \to W_\lambda$ is injective. Suppose not. Then there exists $0 \neq v \in \op{im} D$ such that $\Pi_{W_\lambda}(v) = 0$. Thus, $v \in W_\lambda^\perp = V_\lambda$. Since $V_\lambda = \alpha_\lambda(\ker D^*)$, there exists $\mathfrak o \in \ker D^*$ such that $v = \alpha_\lambda(\mathfrak o)$. Note that $v \in \op{im} D$ implies $v \perp \ker D^*$. Therefore, 
        \begin{align*}
            0 &= \langle \mathfrak o, v \rangle \\
            & = \langle \mathfrak o, \lambda \Pi(\mathfrak o) + (1-\lambda) \mathfrak o \rangle \\
            & = \lambda \sum_{i=1}^k \langle \mathfrak o, \langle \mathfrak o, e_i \rangle e_i \rangle + (1-\lambda) \langle \mathfrak o, \mathfrak o \rangle \\
            & = \lambda \sum_{i=1}^k \langle \mathfrak o, e_i \rangle^2 + (1-\lambda) \langle \mathfrak o, \mathfrak o \rangle,
        \end{align*}
        where $e_1, e_2, \dots, e_k$ is an orthonormal basis of $\widetilde V$. Since $\mathfrak o \neq 0$, we obtain $\lambda = 1$, implying $\mathfrak o \perp \widetilde V$. Hence, $\mathfrak o \in \ker \Pi$, which contradicts the fact that $\Pi$ is injective.
    \end{proof}
    
    \end{remark}
    
    \begin{remark} \label{remark: surjectivity implies transverse}
        For any finite-dimensional $V$ such that the map $\Pi_W D : L^2(E) \to W$ is surjective, we have $V + \op{im} D = L^2(E)$.
    \end{remark}
    \begin{proof}
        Suppose not. Then there exists $0 \neq \mathfrak o \in L^2(E)$ such that $\mathfrak o$ is orthogonal to $\op{im} D + V$. Hence, $\mathfrak o \in \ker D^*$ and $\mathfrak o \in W = V^\perp$. However, since $\mathfrak o$ is orthogonal to $\op{im} D$, $\mathfrak o \notin \Pi_W D$, contradicting the surjectivity assumption.
    \end{proof}

\section{Proof of Theorem~\ref{thm: gluing}} \label{section: proof of theorem}
In this section, we prove Theorem~\ref{thm: gluing}.
The constructions of the section $\obstructionSectionS$ and the gluing map $G$ follow straightforwardly from Section~\ref{section: pregluing} and Proposition~\ref{prop: contraction mapping}, and the map $j_\sharp$ is obtained by dualizing the map $h$ from the Linear Gluing Lemma~\ref{lemma: linear gluing}. The rest of the proof is a specialization of \cite{hutchings2009gluing} to the Morse trajectory case.
However, we choose to provide an alternative proof using semi-global Kuranishi structures, which is also used later to show iterated gluing = simultaneous gluing.

\begin{proof}
\textbf{Step 1.} For any $\delta > 0$, we first define a bundle $\lOc$ over a small neighborhood $\mathcal W$ of the space of close to breaking curves $\lN = \closeToBreaking{\delta}{\compactSubset_+}{\compactSubset_-}$.
For any $[u_\pm] \in \compactSubset_\pm$, we choose representatives $u_\pm$ as in Section~\ref{subsubsection: choice of representative}. 
There exists a large $R$, such that for any $T > R$, 
we can define $w := \lavg(u_+, u_-, T) := \bPlus{T}\uPlus{-2T} + \bMinus{-T}\uMinus{2T}.$
We construct a subspace 
\begin{equation}\label{eqn: direct sum bundle}
    \lOc_{[w]} = \{ \bPlus{T}\xi_{+, -2T} + \bMinus{-T}\xi_{-,2T} \in \mathcal E_w ~|~ \xi_\pm \in \lO_{u_\pm}\} \subset \mathcal E_{[w]},
\end{equation} where $\xi_{\pm, \mp 2T} = \tau_{\mp 2T}^* \xi_\pm \in (\mathcal E_\pm)_{u_{\pm, \mp 2T}}$ is a translation of $\xi_\pm.$ 
The map $\lavg: \compactSubset_+ \times \compactSubset_- \times [R, \infty) \to \mathcal{B}$ is injective, by the uniqueness of solutions of ordinary differential equations. Then the above construction gives a subbundle $\lOc$ over $\op{im}(\lavg) := \lavg(\compactSubset_+ \times \compactSubset_- \times [R, \infty))$. We extend it to a smaller neighborhood $\mathcal W$ as follows: for any $u \in \lW$ sufficiently close to $\op{im}(\lavg)$, there exists a unique closest element $w = \lavg(u_+, u_-, T)$.
Let $\phi \in L_1^2(w^* TM)$ be defined by $\op{Exp}_{w(t)}\phi(t) = u(t)$.
For any $t \in \R$, define $\lOc_{[u]}(t) \in T_{u(t)} M$ by parallel translating $\lOc_{[w]}$ through the curve $\tau \mapsto \op{Exp}_{w(t)} \tau \phi(t)$ for $\tau \in [0,1].$
If $\delta$ is sufficiently small, then $\lW$ also contains $\mathcal N.$

\textbf{Step 2.} Define a thickened moduli space (Kuranishi neighborhood) $\V$ containing $\lN$ by $\V = \morsedbar^{-1}(\lOc)$.
For any $[u] \in \lN$, let $w \in \op{im}(f)$ be the closest element. By Lemma~\ref{lemma: linear gluing} (the map $h$ being surjective), we have $D_{w} \pitchfork_{\mathcal E} \lOc$ is surjective. Since transversality is an open condition, we obtain $D_u \pitchfork_{\mathcal E} \lOc$ by shrinking $\lW$ if necessary. 

Now we state a claim: 
\begin{claim}[Kuranishi gluing]\label{claim: kuranishi gluing}
    There exists a bundle map $(G^\sharp, G): \obstructionBundlePM \to \lOc$ such that:
    \begin{enumerate}[label=(g\arabic*)]
        \item \label{kg1} $G$ is a diffeomorphism onto its image.
        \item \label{kg2} For sufficiently small $\delta$, the image of $G$ contains $\lN = \closeToBreaking{\delta}{\compactSubset_+}{\compactSubset_-}$.
        \item $G^\sharp$ is an isomorphism on fibers.
    \end{enumerate} 
\end{claim} 
Assuming the claim, we proceed with the proof. We restrict the bundle $\lOc$ to $\V$ and still call it $\lOc$. This gives us a commutative diagram:
\begin{equation}\label{fig: thicken gluing diagram}
\begin{tikzcd}
\obstructionBundlePM \arrow[d] \arrow[r, "G^\sharp"] & \lOc \arrow[d]  \\
\openSubset_+ \times \openSubset_- \times [R,\infty) \arrow[u, bend left, "\obstructionSectionS"] \arrow[r, "G"] & \V \arrow[u, swap, bend right, "\morsedbar"]
\end{tikzcd}.
\end{equation}

\textbf{Step 3.} We define $\obstructionSectionS = (G^\sharp)^{-1} \circ \morsedbar \circ G$.  
Since, by definition, $G^\sharp \circ \obstructionSectionS = \morsedbar \circ G$, we have $G (\obstructionSectionS^{-1}(0)) \subset \morsedbar^{-1}(0) \subset \M$.
The gluing map $\obstructionSectionS^{-1}(0) \to \M$ in the statement of the theorem is simply the restriction of $G$ to $\obstructionSectionS^{-1}(0)$.
Let $\op{pr}: \obstructionBundle \to \lOc$ be the orthogonal projection, 
which is injective by Remark~\ref{remark: linear gluing}.

Define the map $j^\sharp = (G^\sharp)^{-1} \circ \op{pr}$.
Note that $\morsedbar \pitchfork \op{pr}(\lO)$ by Claim~\ref{claim: lambda surjective} and Remark~\ref{remark: surjectivity implies transverse}.
Hence, we have $i^\sharp \circ j^\sharp$ transverse to $\obstructionSectionS$.
The statement that $\obstructionSectionS$ intersects the zero section cleanly is equivalent to $\morsedbar^{-1}(0) = \morsedbar^{-1}(\obstructionBundle)$ by Lemma~\ref{lemma:clean-intersection-criterion}, which holds by the definition of $\obstructionBundle$ and the assumption that $\M$ is cleanly cut out.

\begin{proof}[Proof of Claim]
This is the standard gluing theorem for Kuranishi structures. See, for instance, Section 7 of \cite{bao2015semi}. We outline the proof below.
Let $\lH_\pm = (\ker D_{u_\pm})^\perp$, and $\psi_\pm \in \lH_\pm$.
To construct the gluing map $G$, we first define a pregluing $u$ as in Equation~\eqref{eqn: pregluing}. 
Then $[u] \in \V$ satisfies the equation $\morsedbar u \in \lOc_{[u]}$, or more precisely, $\Phi \morsedbar u \in \Phi \lOc_{[u]}$
which becomes 
\begin{equation} \label{eqn: kuranishi equation for Theta 2T}
    \bPlus{T}\Theta_+^{-2T} + \bMinus{-T}\Theta_-^{2T} \in \Phi \lOc_{[u]},
\end{equation} 
where $\Phi: u^*TM \to w^*TM$ is the identification defined in Section~\ref{section: pregluing}, $w = \bPlus{T}\uPlus{-2T} + \bMinus{-T}\uMinus{2T}$, and $\Theta_\pm^{\mp 2T}$ are defined in Equations~\eqref{eqn: Theta plus 2T} and \eqref{eqn: Theta minus 2T}.
We define two vector spaces $\lOc^+_{[u]} = \{ \frac{\bPlus{T}}{\bPlus{T} + \bMinus{T}} \eta ~|~ \eta \in \lOc_{[u]}\}$
and $\lOc^-_{[u]} = \{ \frac{\bMinus{T}}{\bPlus{T} + \bMinus{T}} \eta ~|~ \eta \in \lOc_{[u]}\}$.
The space $\lOc^+_{[u]}$ can be viewed as a subspace of $\mathcal{E}_{[\uPlus{-2T} + \pPlus{-2T}]}$, and 
the space $\lOc^-_{[u]}$ can be viewed as a subspace of $\mathcal{E}_{[\uMinus{2T} + \pMinus{2T}]}$. 

To solve Equations~\eqref{eqn: kuranishi equation for Theta 2T}, it is sufficient to solve the system 
\[
\Theta_\pm^{\mp 2T}(\psi_+, \psi_-) \in \Phi_\pm \lOc^\pm_{[u]},
\]
where $\Phi_\pm: (\op{Exp}_{u_{\pm, \mp 2T}} \psi_{\pm, \mp 2T})^* TM \to u_{\pm, \mp 2T}^* TM$ is the identification defined in Section~\ref{section: pregluing}.
Applying translation $\trans{\mp 2T}$ and denoting $\lOc_\pm':= \trans{\mp 2T} \Phi_\pm \lOc^\pm_{[u]} \subset \mathcal{E}_{u_\pm}$, we get 
\[
\Theta_\pm (\psi_+, \psi_-) \in \lOc_\pm',
\]
 i.e.,
\[
    D_{u_\pm} \psi_\pm + \frac{d}{ds}(\bump{\mp}{\pm T} (u_{\mp,4T} + \psi_{\mp,4T}))   \in  \lOc_\pm'.
\]
Let $\Pi_\pm: \mathcal{E}_{u_\pm} \to (\lOc_\pm')^\perp$ be the $L^2$-orthogonal projection. Then the equations become 
\begin{equation}\label{eqn: kuranishi gluing projected}
    \Pi_\pm (D_{u_\pm} \psi_\pm + \frac{d}{ds}(\bump{\mp}{\pm T} (u_{\mp,4T} + \psi_{\mp,4T}))) = 0.
\end{equation}
But since $\lOc_\pm'$ is close to $\lO_\pm$ when $T$ is large and $D_{u_\pm}$ is transverse to $\lO_\pm$, it follows that $\Pi_\pm D_{u_\pm}$ is surjective.

Similar to the proof of Proposition~\ref{prop: contraction mapping}, the contraction mapping theorem implies that for each $(u_+, u_-, T) \in U_+ \times U_- \times [R, \infty)$, let $\lH_\pm = (\ker \Pi_\pm D_{u_\pm})^\perp$.
Then there exists a unique solution $(\Perturb_+, \Perturb_-) \in \lH_+ \times \lH_-$ that solves Equations~\eqref{eqn: kuranishi gluing projected}, and hence we obtain the gluing map $G$ mapping $(u_+, u_-, T)$ to $u$ as in Claim~\ref{claim: kuranishi gluing}. Statements~\ref{kg1} and \ref{kg2} then follow from the standard gluing result for Kuranishi structures.
The map $G^\sharp$ is defined as follows: for any $\eta_+ \in {\obstructionBundle_+}_{[u_+]} \subset \mathcal{E}_{[u_+]}$, 
we first translate the domain and obtain $\eta_{+,-2T} := \trans{-2T}^* \eta_+ \in \mathcal{E}_{[\uPlus{-2T}]}$,
and then we can parallel transport $\eta_{+,-2T}$ to $\eta_{+,-2T}' \in \mathcal{E}_{[\uPlus{-2T} + \pPlus{-2T}]}$ via the path $\tau \mapsto \op{exp}_{\uPlus{-2T}}(\tau \pPlus{-2T})$ for $\tau \in [0,1]$. 
Similarly, for any $\eta_- \in {\obstructionBundle_-}_{[u_-]}$, we get $\eta_{-,2T}'$. Then we define the image of $(\eta_+, \eta_-)$ to be $\bPlus{T}\eta_{+,-2T}' + \bMinus{-T}\eta_{-,2T}'$ projected orthogonally to $\lOc$.
By construction, it is clear that $G^\sharp$ is an isomorphism. This completes the proof of the claim.
\end{proof}
\end{proof}

\begin{remark} 
From the proof, it is evident that the theorem remains valid even if the bundle $\obstructionBundle_\pm$ is replaced by other bundles whose fibers are given by $V_\pm$, as in linear gluing Lemma~\ref{lemma: linear gluing}, provided that they satisfy the surjectivity assumption stated therein. In particular, by Claim~\ref{claim: lambda surjective}, we can select the fibers to be $V_\lambda$ near the boundary of the moduli space.
\end{remark}

\section{Iterated Gluing vs. Simultaneous Gluing}
In this section, we show that without adjustments difference between iterated gluiing and simultaneous gluing goes to zero as the gluing parameters goes to zero. In Section~\ref{section: construction of semi-global} we make adjustments so that the difference is zero as stated in \ref{k2}.
We order all moduli spaces $$\{\modF{p}{q}~|~ p,q \in \crit{f}, f(p) > f(q)\}$$ as $\M_1, \M_2, \dots$ such that the energy $f(p) - f(q)$ of $\modF{p}{q}$ increases.

We define the source of the moduli space $\modF{p}{q}$ as $p$ and denote it by $\source(\modF{p}{q}) = p$, and the target of the moduli space $\modF{p}{q}$ as $q$, denoted by $\target(\modF{p}{q}) = q$.

We say $\iT{I} = \indexTuple{I}$ is an index tuple of length $n(\iT{I})$ if $i_m \in \Z_{>0}$ for $m \in \{1, 2, \dots, n(\iT{I})\}$ and $\target(\M_{i_m}) = \source(\M_{i_{m+1}})$ for $m \in \{1, 2, \dots, n(\iT{I})-1\}$. The contraction $\contraction{\iT{I}}{\iT{I}}$ of $\iT{I}$ is defined as the integer $\ell$ such that $\source(\M_{\ell}) = \source(\M_{I_1})$ and $\target(\M_{\ell}) = \target(\M_{I_i})$.

More generally, a \emph{sub-index tuple} $\iT{S}$ of $\iT{I}$ is a sub-tuple of $\iT{I}$ that is also an index tuple. We can contract $\iT{I}$ along a sub-index tuple $\iT{S}$ by replacing the sub-tuple $\iT{S}$ with the integer $\contraction{S}{S}$ and denote the resulting index tuple by $\contraction{\iT{I}}{\iT{S}}$.

We say $\SIC{S}$ is a sub-index collection of $\iT{I}$, written as $\SIC{S} \subseteq \iT{I}$, if $\SIC{S} = (\iT{S}^1, \dots, \iT{S}^j)$ for some $j$ such that $\iT{S}^1, \dots, \iT{S}^j$ are non-overlapping sub-index tuples of $\iT{I}$. We can also contract along a \emph{sub-index collection} in the obvious way. We denote the contraction of $\iT{I}$ along $\SIC{S}$ by $\contraction{\iT{I}}{\SIC{S}}$. We say index tuple $\iT{J} \leq \iT{I}$ if $\iT{J} = \contraction{\iT{I}}{\SIC{S}}$ for some sub-index collection $\SIC{S}$ of $\iT{I}$.

Let $\iT{I} = (i_1, i_2, \dots, i_{\tt n})$ be an index tuple with $\ell = \contraction{I}{I}$. Let $q_0, q_1, \dots, q_{\tt n}$ be distinct critical points of $f$ such that $\target(\M_{i_k}) = q_k$ and $\source(\M_{i_{k+1}}) = q_k$ for $k \in \{1, 2, \dots, \tt n\}$. Let $\compactSubset_{i_k} \subset \M_{i_k}$ be compact subspaces for $k \in \{1, 2, \dots, \tt n\}$. Let $\delta > 0$ be a small constant.

\begin{definition}[Close to breaking (multiple buildings)]\label{defn:close_to_breaking}
A map $[u] \in \M_\ell$ is \emph{$\delta$-close to breaking into the broken trajectory} $([u_{i_1}], [u_{i_2}], \dots, [u_{i_{\tt n}}]) \in \compactSubset_{i_1} \times \compactSubset_{i_2} \times \dots \times \compactSubset_{i_{\tt n}}$ if there exist representatives $u, u_{i_1}, u_{i_2}, \dots, u_{i_{\tt n}}$ of $[u], [u_{i_1}], [u_{i_2}], \dots, [u_{i_{\tt n}}]$ respectively, and constants $a_1 < b_1 < \dots < a_{\tt n} < b_{\tt n}$ such that for each $k \in \{1, 2, \dots, \tt n\}$, the following hold:
\begin{enumerate}[label=(\alph*)]
    \item $u|_{[a_k, b_k]}$ and $u_k|_{[a_k, b_k]}$ are $\delta$-close in $C^1$-norm.
    \item $u|_{(b_k, a_{k+1})}$ is $\delta$-close in $C^1$-norm to the constant map $q_k$, where $b_0 := -\infty$ and $a_{\tt n} := \infty$;
    \item $u_k|_{(-\infty, a_k)}$ is $\delta$-close in $C^1$-norm to the constant map $q_{k-1}$.
    \item $u_k|_{(b_k, \infty)}$ is $\delta$-close in $C^1$-norm to the constant map $q_k$.
\end{enumerate}
Denote $\compactSubset_{\iT{I}} = \compactSubset_{i_1} \times \compactSubset_{i_2} \times \dots \times \compactSubset_{i_{\tt n}}$, and the space of such $\delta$-close to breaking maps by $\closeToBreakings{\delta}{I}$.
\end{definition}

Let $R > 0$ be a real number, which we call the gluing parameter bound. Let $\openSubset_{i_k}$ be an open neighborhood of $\compactSubset_{i_k}$, i.e., $\compactSubset_{i_k} \subset \openSubset_{i_k} \subset \M_{i_k}$ for $k \in \{1, 2, \dots, \tt n\}$. We denote by $\productModuliSpaces{I}{R}$ the product $\openSubset_{i_1} \times_R \openSubset_{i_2} \times_R \dots \times_R \openSubset_{i_{\tt n}}$, where $A \times_R B$ is the abbreviation of $A \times [R, \infty) \times B$. We also denote by $\mathbb K_I^R$ the product $\compactSubset_{i_1} \times_R \compactSubset_{i_2} \times_R \dots \times_R \compactSubset_{i_{\tt n}}$. Let $\obstructionBundle_{\ell} \to \M_{\ell}$ be the obstruction bundle, i.e., the fiber is $\ker D_u^*$. We write $\obstructionBundle_{\ell} \to \closeToBreakings{\delta}{I}$ for the bundle $\obstructionBundle_{\ell}|_{\closeToBreakings{\delta}{I}} \to \closeToBreakings{\delta}{I}$. We denote by $\productBundle{I}$ the obstruction bundle over the product $\productModuliSpaces{I}{R}$ defined by $\productBundle{I} = \oplus_{j = 1}^{i} \op{pr}_k^* \obstructionBundle_{i_k}$, where $\op{pr}_k: \productModuliSpaces{I}{R} \to \openSubset_{i_k}$ is the projection map.

\begin{theorem}[Simultaneous gluing]\label{thm: simultaneous gluing}
    Let $\iT{I} = (i_1,\dots, i_{\tt n})$ be an index tuple of length $\tt n$ such that $\contraction{\iT{I}}{\iT{I}} = \ell$.
    Suppose that $\M_{i_1},\dots, \M_{i_{\tt n}}$, and $\M_{\ell}$ are cleanly cut out.
    For any compact subsets $\compactSubset_{i_k} \subset \M_{i_k}$ and their sufficiently small open neighborhoods $\compactSubset_{i_k} \subset \openSubset_{i_k} \subset \M_{i_k}$ for $k \in \{1, 2, \dots, \tt n\}$,
    there exist a gluing parameter bound $R > 0$, a close-to-breaking parameter bound $\delta > 0$, and a $C^1$-section $\obstructionSection{I}$, called the obstruction section, of the bundle $\productBundle{I} \to \productModuliSpaces{I}{R}$,
    such that the following holds:
    \begin{enumerate}[label = (G\arabic*)]
        \item The obstruction section $\obstructionSectionS$ is a $C^1$-map that intersects the zero section cleanly. We denote $\gluableSet = \obstructionSection{I}^{-1}(0)$. Let $i: \gluableSet \to \productModuliSpaces{I}{R}$ be the inclusion map.
        \item There exists a $C^1$-gluing map $G_{\iT{I}}: \gluableSet \to \M_\ell$ such that:
            \begin{enumerate}[label=(\roman*)]
                \item The map $G_{\iT{I}}$ is a diffoemorphism onto its image.
                \item For any $R' \geq R$, there exists $\delta' > 0$ such that $\closeToBreakings{\delta'}{I} \subset G_{\iT{I}} \circ i^{-1}(\productModuliSpaces{I}{R'}).$
                \item For any $0 < \delta' < \delta$, there exists $R'$ such that $G_{\iT{I}} \circ i^{-1}(\mathbb K_I^{R'} ) \subset \closeToBreakings{\delta'}{I}.$
            \end{enumerate}
        \item There exists a bundle map $j^\sharp: G_{\iT{I}}^* \obstructionBundle_\ell \to i^* \productBundle{I}$ such that:
            \begin{enumerate}[label=(\roman*)]
                \item $j^\sharp$ is injective.
                \item $i^\sharp \circ j^\sharp$ is transverse to $\productBundle{I}$, where $i^\sharp$ is the induced bundle map $i^* \productBundle{I} \to \productBundle{I}$. See the following diagram:
            \end{enumerate}
            \begin{equation}\label{fig: simul gluing diagram}
            \begin{tikzcd}
                \productBundle{I} \arrow[d] & i^* \productBundle{I} \arrow{dr} \arrow[l, "i^\sharp"] & & G^* \obstructionBundle_\ell \arrow[ll, "j^\sharp"] \arrow{dl} \arrow[r, "G^\sharp"] & \obstructionBundle_\ell \arrow[d] \\
                \productModuliSpaces{I}{R} \arrow[u, bend left, "\obstructionSection{I}"] & & \gluableSet \arrow[ll, "i"] \arrow[rr, "G"] & & \M_\ell \\
            \end{tikzcd}.
            \end{equation}
        \item For any $m\in\{1,\dots, n(\iT{I})-1\}$, let $T_{m} \in [R,\infty)$ be the $m$-th factor of $[R,\infty)$ in $\productModuliSpaces{I}{R}$. Then $\obstructionSection{I}\to (\obstructionSection{I_1}, \obstructionSection{I_2})$ in $C^1$-norm as $T_m \to \infty$, where $\iT{I} = \indexTuple{I}$, $\iT{I_1} = (i_1,\dots, i_m)$ and $\iT{I_2} = (i_{m+1},\dots, i_{n(\iT{I})}).$ Here we set the notation $\obstructionSection{I'} = 0$, if $n(\iT{I'}) = 1.$    
    \end{enumerate}
\end{theorem}

Similar as Theorem~\ref{thm: gluing}, there are two ways to prove Theorem~\ref{thm: simultaneous gluing}: a straightforward obstruction bundle gluing, or a proof using Kuranishi gluing. Both methods are straightforward generalizations of the proofs of Theorem~\ref{thm: gluing}.

\begin{proof}[Sketch of proof]
The proof is similar to that of Theorem~\ref{thm: gluing}. For details on the Kuranishi case in the context of pseudo-holomorphic curves, see the proof of Theorem 6.5.1 in \cite{bao2015semi}. Following this, define $\obstructionSection{I} = (G_I^\sharp)^{-1} \circ \morsedbar \circ G_I$ as in Step 2 of the proof of Theorem~\ref{thm: gluing}.
\end{proof}

Let $\breakingBase{I} = i \circ G^{-1}: \closeToBreakings{\delta}{I} \to \productModuliSpaces{I}{R}$, and let $B^\sharp = i^\sharp \circ j^\sharp \circ (G^\sharp)^{-1}: \productBundle{I}|_{\closeToBreakings{\delta}{I}} \to \productBundle{I}$.
We rephrase the theorem as follows.

\begin{corollary} 
Under the same assumptions as Theorem~\ref{thm: simultaneous gluing},
there exist a gluing parameter bound $R > 0$, a constant $\delta > 0$,
and a $C^1$-bundle map $(\breakingBundle{\iT{I}}, \breakingBase{\iT{I}})$: 
\begin{equation}\label{fig: simplified gluing diagram}
\begin{tikzcd}
\productBundle{I} \arrow[d] & \obstructionBundle_{\ell} \arrow[d] \arrow[l, "\breakingBase{\iT{I}}^\sharp"] \\
\productModuliSpaces{I}{R} \arrow[u, bend left, "\obstructionSection{I}"] & \closeToBreakings{\delta}{I} \arrow[l, "\breakingBase{\iT{I}}"]\\
\end{tikzcd}
\end{equation}
such that 
\begin{enumerate}
    \item $\breakingBase{\iT{I}}$ (the inverse of the gluing map) is diffeomorphic onto its image,  
    \item $\obstructionSection{I} \circ \breakingBase{\iT{I}} = 0$,
    \item for any $R' > R$, there exists $\delta'> 0$ such that $\breakingBase{I}(\closeToBreakings{\delta'}{I}) \subset \productModuliSpaces{I}{R'}$,
    \item for any $0 < \delta' < \delta$, there exists $R'$ such that $\mathbb K_I^{R'} \subset \breakingBase{I}(\closeToBreakings{\delta'}{I})$,
    \item $\breakingBundle{\iT{I}}$ is injective,
    \item $\obstructionSection{I}$ is transverse to $\breakingBundle{I}$,
    \item for any $m\in\{1,\dots, n(\iT{I})-1\}$, let $T_{m} \in [R,\infty)$ be the $m$-th factor of $[R,\infty)$ in $\productModuliSpaces{I}{R}$. Then $\obstructionSection{I}\to (\obstructionSection{I_1}, \obstructionSection{I_2})$ in $C^1$-norm as $T_m \to \infty$, where $\iT{I} = \indexTuple{I}$, $\iT{I_1} = (i_1,\dots, i_m)$ and $\iT{I_2} = (i_{m+1},\dots, i_{n(\iT{I})}).$ 
\end{enumerate}
\end{corollary}

Given a broken curve that consists of $a$ smooth curves, we can first glue the middle $b$ ($b < a$) curves to get one curve and then glue the remaining $b - a+1$ curves, or we can glue them simultaneously. 
We will show that both methods of gluing result in the same curve.

Let $\iT{I} = \indexTuple{I}$ be an index tuple and $\iT{J} = \indexTuple{J}$ be another index tuple with $j_m = \contraction{I}{I}$, for some $m \in \{1,2, \dots, n(\iT{J})\}$, and denote $c = \contraction{J}{J}.$
Let $\iT{K}$ be the index tuple that replaces $j_m$ with $\iT{I}$.
In other words, 
$$\iT{K} = \indexTuple{K} = (j_1,\dots, j_{m-1}, \underbrace{i_{1}, \dots, i_{n(\iT{I})}}_{\iT{I}}, j_{m+1}, \dots j_{n(\iT{J})}).$$

Iterated gluing that is represented by $\contraction{(\contraction{\iT{K}}{\iT{I}})}{(\contraction{\iT{K}}{\iT{I}})}$ equals simultaneous gluing that is represented by $\contraction{\iT{K}}{\iT{K}}$.

Since the gluing construction is not canonical, in general iterated gluing does not equal simultaneous gluing. But their difference is asymptotically small with respect to the gluing parameters. In the next section, in the construction of the minimal semi-global Kuranishi structure, we perturb the gluing maps so that they are the same.

\begin{theorem}[Iterated gluing vs. simultaneous gluing] 
Let $\iT{I}, \iT{J}, \iT{K}$ be index tuples as above such that $\contraction{K}{I} = \iT{J}$, $\contraction{I}{I} = j_m$, and $\contraction{K}{K} = c$. 
Suppose the moduli spaces $\M_i$ for $i \in \iT{I} \cup \iT{J} \cup \{c\}$ are cleanly cut out. 
For any $i \in \iT{I} \cup \iT{J}$ and for any compact subsets $\compactSubset_{i} \subset \M_i$ ,
let $U_i$ be their open neighborhoods,
let $R_\dagger$ be the gluing parameter bounds,
and let $\delta_\dagger$ be the close-to-breaking parameter bounds,
such that the breaking maps $\breakingBase{\dagger}: \closeToBreakings{\delta_\dagger}{\dagger} \to  \productModuliSpaces{\dagger}{R_\dagger}$, for $\dagger \in \{\iT{I},\iT{J}, \iT{K}\}$ are defined as in Theorem~\ref{thm: simultaneous gluing}.
For any $0 < \delta < \delta_{J}$, such that $$\pi_{{j_m}} \circ \breakingBase{J}(\closeToBreakings{\delta_K}{K} \cap \closeToBreakings{\delta}{J}) \subset \closeToBreakings{\delta_I}{I},$$ where $\pi_{{j_m}}: \productModuliSpaces{J}{R} \to U_{j_m}$ is the projection,
we have the following gluing diagram:
\begin{equation}
    \begin{tikzcd}
    \productBundle{K}  \arrow[d]  & &  \obstructionBundle_c \arrow[ll, "\breakingBundle{\iT{K}}"] \arrow[d] \arrow[ldd, pos=0.65, "\breakingBundle{\iT{J}}"]   \\
    \productModuliSpaces{K}{R_K}  \arrow[u, bend left, "\obstructionSection{K}"]  & & \closeToBreakings{\delta_K}{K} \cap \closeToBreakings{\delta}{J} \arrow[ll, "\breakingBase{\iT{K}}"] \arrow[ldd, "\breakingBase{\iT{J}}"]   \\ 
    & \productBundle{J} \arrow[d] \arrow[luu, near start, "\breakingBundle{\iT{I}}\times \op{id}"]  & \\ 
    & {\mathbb U}_{\iT{J}}^{R_J} \cap_m \closeToBreakings{\delta_I}{I} \arrow[luu, "\breakingBase{\iT{I}}\times \op{id}"] \arrow[u, bend left, "\obstructionSection{J}"]  &  \\
    \end{tikzcd},
\end{equation} 
where 
\begin{itemize} 
    \item ${\mathbb U}_{\iT{J}}^{R_J} \cap_m \closeToBreakings{\delta_I}{I}$ is the subspace of $\productModuliSpaces{J}{R_J}$ defined by 
    intersecting the $m$-th factor of $\productModuliSpaces{J}{R_J}$ by $\closeToBreakings{\delta_I}{I}$, i.e.,
    \[
         \openSubset_{j_1}\times_R \dots \times_R \openSubset_{j_{m-1}} \times_R ( \openSubset_{j_m} \cap \closeToBreakings{\delta_I}{I}) \times_R \openSubset_{j_{m+1}} \times_R  \dots \times_R \openSubset_{j_n(\iT{K})},
    \]
    for $R = R_J$,
    \item $\breakingBase{\iT{I}}\times \op{id}$ is the map that equals $\breakingBase{\iT{I}}$ on the factor $\openSubset_{j_m} \cap \closeToBreakings{\delta_I}{I}$ and the identity map on the other factors,

    \item $\breakingBundle{\iT{I}} \times \op{id}$ is the map that equals $\breakingBundle{\iT{I}}$ on the summand $\obstructionBundle_{j_m}$ and the identity map on the other summands,

    \item the two horizontal maps correspond to the simultaneous (inverse of) gluing, and
    
    \item the lower right two maps correspond to the gluing of the curves labeled by ${\iT{J}}$.

\end{itemize}
The diagram commutes asymptotically in the following sense:
\begin{enumerate}
    \item The iterated breaking $ (\breakingBundle{\iT{I}} \times \op{id}, \breakingBase{\iT{I}}\times \op{id}) \circ (\breakingBundle{\iT{J}}, \breakingBase{\iT{J}})$ is $C^1$-close to the simultaneous breaking $(\breakingBundle{\iT{K}}, \breakingBase{\iT{K}})$, and difference goes to $0$ as $\min_{\ell = 1}^{n(\iT{K})-1} T_\ell \to \infty$, where $T_\ell \in [R,\infty)$ is the $\ell$-th gluing parameter;
    \item The maps $(\breakingBundle{\iT{I}}\times \op{id}) \circ  \obstructionSection{J}$ and $\obstructionSection{K} \circ (\breakingBase{\iT{I}}\times \op{id})$ are $C^1$-close, and their difference goes to $0$, as $\min_{\ell = 1}^{n(\iT{K})-1} T_\ell \to \infty$.
\end{enumerate}

\end{theorem}
\begin{proof}[Sketch of proof]
This essentially follows from Theorem 6.5.3 in \cite{bao2015semi}, which states that the difference between iterated gluing and simultaneous gluing approaches zero in the Kuranishi case in the context of pseudo-holomorphic curves. This proves (1), (2), and (3)(a); (3)(b) follows automatically from (3)(a) and the definition of the obstruction sections.
\end{proof}

\begin{remark} \label{remark: iterated gluing vs simultaneous gluing}
Note that, in Kuranishi case where the moduli spaces are thickened, one can expect that if a broken trajectory glues simultaneously, it will also glue iteratively. However, in the case of obstruction bundle gluing, this is not always true. In Example~\ref{example: upright torus}, two of the broken trajectories in $\M(p,q) \times \M(q,r) \times \M(r,s)$ can glue simultaneously if the two gluing parameters at $q$ and $r$ are the same, but they never glue iteratively, since $\M(p,r) = \emptyset$ and $\M(q,s) = \emptyset$.

\end{remark}

\section{Definition of minimal semi-global Kuranishi structures}\label{section: definition of minimal semi-global}
Suppose all the moduli spaces $\M_1, \M_2, \dots $ are cleanly cut out. In this section we provide the definition of minimal semi-global Kuranishi structures, and their perturbation sections.

\begin{definition}[Minimal semi-global Kuranishi structure for clean intersection] \label{defn: clean semi-global Kuranishi}
A minimal semi-global Kuranishi structures for moduli spaces $(\overline{\M}_1,\overline{\M}_2,\dots)$ consists of vector bundles:
$$\obstructionBundle_1 \to \M_1, \obstructionBundle_2 \to \M_2, \dots$$ with $\op{rank} \obstructionBundle_i =  \dim \M_i - \op{virdim} \M_i$,
and the data $$\kur_1, \kur_2, \dots:$$
\begin{enumerate}[label = (K\arabic*)]
 
    \item \label{k1}$\kur_\ell = \{ \chartData{I} ~|~ \contraction{\iT{I}}{\iT{I}} = \ell \}$ with $\chartData{I} = (\kChart{I},  \breakingBundle{I}, \breakingBase{I}, \obstructionSection{I})$ such that:
    \be 
        \item $\kChart{I}$ is an open subset of $\M_\ell$. When $\iT{I} = (\ell)$, we simply write $\interiorKChart{\ell} =  \kChart{(\ell)}$, and call it the interior chart. 
        
        \item $\M_\ell = \cup_{\contraction{I}{I} = \ell} \kChart{I}.$ Furthermore, we require $\kChart{I}\cap \kChart{J} \neq \emptyset$ if and only if either $J\leq I$ or $I\leq J$. 

    \ee 
    If $\iT{I} \neq (\ell)$, we have the following:
    \be[start = 3]
        \item $\productCharts{I} = \interiorKChart{i_1}\times_R \dots \times_R \interiorKChart{i_{n(\iT{I})}},$ for some $R = R(I)>0$.
        
        \item $\productBundle{I} = \op{pr}_{i_1}^* \obstructionBundle_{i_1}\oplus \dots \oplus \op{pr}_{i_{n(\iT{I})}}^* \obstructionBundle_{i_{n(\iT{I})}}$, where $\op{pr}_{i_m}: \productCharts{I} \to \kChart{i_m}$ is the projection map.

        \item $\obstructionSection{I}: \productCharts{I} \to \productBundle{I}$ is a section, a.k.a., the obstruction section, and $\obstructionSection{I} = 0$ if $n(\iT{I}) = 1$.
        
        \item $(\breakingBundle{\iT{I}}, \breakingBase{\iT{I}})$ is a bundle map, 
        \begin{equation*}
            \begin{tikzcd}
            \productBundle{I} \arrow[d] & \obstructionBundle_{\ell}|_{\kChart{I}} \arrow[d] \arrow[l, "\breakingBase{\iT{I}}^\sharp"] \\
           \productCharts{I} \arrow[u, bend left, "\obstructionSection{I}"] & \kChart{I} \arrow[l, "\breakingBase{\iT{I}}"]
           \end{tikzcd},
        \end{equation*}
        such that:
        \be
            \item $\breakingBundle{I}$ is injective.
            \item $\obstructionSection{I}$ intersects the zero section cleanly.
            \item $\breakingBase{I}$ maps $\kChart{I}$ diffeomorphic to $\obstructionSection{I}^{-1}(0)$.
            \item $\obstructionSection{I}$ is transverse to $\breakingBundle{I}$.
            \item for any $p\in\{1,\dots, n(\iT{I})-1\}$, let $T_{p} \in [R(I),\infty)$ be the $p$-th factor of $[R(I),\infty)$ in $\productCharts{I}$. Then $\obstructionSection{I}\to (\obstructionSection{J}, \obstructionSection{K})$ in $C^1$-norm as $T_p \to \infty$, where $\iT{I} = \indexTuple{I}$, $\iT{J} = (i_1,\dots, i_p)$, and $\iT{K} = (i_{p+1},\dots, i_{n(\iT{I})}).$ 
        \ee 
    \hspace{-0.7cm}If $\iT{I} = (\ell)$, then $\productCharts{I} =  \kChart{I}$, $\productBundle{I} = \obstructionBundle_\ell$, $(\breakingBundle{I}, \breakingBase{I}) = (\op{id},\op{id})$, and $\obstructionSection{I} = 0$.
    \ee 
    \item \label{k2} For each pair of index tuples $(\iT{J}, \iT{K})$ with $\contraction{\iT{J}}{\iT{J}} = \contraction{\iT{K}}{\iT{K}} = c,$ $\iT{J} \leq \iT{K}$, the pairs $(\breakingBase{J},\breakingBundle{J})$ and $(\breakingBase{K},\breakingBundle{K})$ satisfy the following compatibility condition:
    \be 

        \item Case when $\iT{J} = \contraction{K}{I}$ for some sub-index tuple $\iT{I}$ of $\iT{K}$. Suppose that $\iT{J} = \indexTuple{J}$ with $j_m = \ell$,
        and $\iT{I} = \indexTuple{I}$ such that $\contraction{I}{I} = \ell$.
        Let $\iT{K}$ be the index tuple that replaces $j_m$ with $\iT{I}$. In other words, 
        $$\iT{K} = (j_1,\dots, j_{m-1}, \underbrace{i_{1}, \dots, i_{n(\iT{I})}}_{\iT{I}}, j_{m+1}, \dots j_{n(\iT{J})}).$$
        Consider the subspace $\productChartsIntersection{K}{J}\subset \productCharts{J}$ defined by 
        \[
            \productChartsIntersection{K}{J}  :=  \interiorKChart{j_1}\times_R \dots \times_R \interiorKChart{j_{m-1}} \times_R ( \interiorKChart{j_m} \cap \kChart{I} ) \times_R \interiorKChart{j_{m+1}} \times_R \dots \times_R \interiorKChart{j_{n(\iT{J})}},
        \]
        where $R = R(J)$ as in \ref{k1},
        and the bundle map $(\breakingBundle{\iT{I}}\times \op{id}, \breakingBase{\iT{I}}\times \op{id})$ from $\productBundle{J} \to \productChartsIntersection{K}{J}$ to $\productBundle{K} \to \productCharts{K}$.
        Then $\breakingBase{J}(V_K \cap V_J) \subset \productChartsIntersection{K}{J}$ and we have Diagram~\ref{eqn: iterrated vs simultaneous}
        \begin{figure*}
        \begin{equation}\label{eqn: iterrated vs simultaneous}
            \begin{tikzcd}
            \productBundle{K}  \arrow[d]  &  &  \obstructionBundle_{c} \arrow[d] \arrow[ldd, pos=0.65, "\breakingBundle{\iT{J}}"]  \arrow[ll, swap, "\breakingBundle{\iT{K}}"]  \\
            \productCharts{K}  \arrow[u, bend left, "\obstructionSection{K}"]  &  & \kChart{K} \cap \kChart{J}  \arrow[ldd, "\breakingBase{\iT{J}}"]   \arrow[ll, swap, "\breakingBase{\iT{K}}"] \\ 
            & \productBundle{J} \arrow[d] \arrow[luu, swap, pos = 0.7, "\hspace{-0.2cm} \breakingBundle{\iT{I}}\times \op{id}"]  &\\ 
            & \productChartsIntersection{K}{J} \arrow[luu, pos = 0.7, "\breakingBase{\iT{I}}\times \op{id}"] \arrow[u, bend left, pos = 0.8, "\obstructionSection{J}"] \arrow[r, hook] &  \productCharts{J} \\
            \end{tikzcd}
        \end{equation} 
    \end{figure*}
        satisfying:
        \begin{enumerate}
            \item $(\breakingBundle{\iT{K}}, \breakingBase{\iT{K}}) = (\breakingBundle{\iT{I}} \times \op{id}, \breakingBase{\iT{I}}\times \op{id}) \circ (\breakingBundle{\iT{J}}, \breakingBase{\iT{J}})$.
            \item $(\breakingBundle{\iT{I}}\times \op{id}) \circ  \obstructionSection{J} =  \obstructionSection{K} \circ (\breakingBase{\iT{I}}\times \op{id}).$ 
        \end{enumerate}
        \item Case when $\iT{J} = \contraction{\iT{K}}{\SIC{I}}$ for some sub-index collection $\SIC{I}$. We have similar bundle maps and commutative diagrams as Formula~\ref{eqn: iterrated vs simultaneous} with the straightforward modification of $\productChartsIntersection{K}{J}$ replacing $\iT{I}$ replaced with $\SIC{I}$.
    \ee 
\end{enumerate}
\end{definition}
\begin{remark}
    The requirement \ref{k1}(b) 
    \[
    \kChart{I}\cap \kChart{J} \neq \emptyset \text{ if and only if either } J\leq I \text{ or } I\leq J
    \]
    is not merely for simplicity. It is used in the construction of minimal semi-global Kuranishi structures in Section~\ref{section: construction of semi-global}. More importantly, it is necessary (see Remark~\ref{remark: iterated gluing vs simultaneous gluing}).
\end{remark}





\begin{definition}[Compatible sections for $\kur$] 
Given a minimal semi-global Kuranishi structure $(\kur_1, \kur_2, \dots )$ for $(\overline{\M}_1, \overline{\M}_2, \dots)$, 
we call $(\kuranishiSectionBold_1, \kuranishiSectionBold_2, \dots)$ a \emph{compatible section} if
$\kuranishiSectionBold_\ell =\{\kuranishiSection{I}\}_{\contraction{I}{I} = \ell}$ for $\ell = 1, 2, \dots$, where $\sigma_I$ are sections of $\productBundle{I} \to \productCharts{I}$ that satisfy:
    \begin{enumerate}[label = ($\bold P$\arabic*)]
        \item (Compatibility with other charts)\label{P1} For any $\iT{J} \leq \iT{K}$ as in \ref{k2}, $\kuranishiSection{K}$ and $\kuranishiSection{J}$ agree on their overlaps, i.e., the following diagram commutes:
        \begin{equation}
            \begin{tikzcd}
            \productBundle{K} \arrow[d] & \productBundle{J} \arrow[l, swap, "\breakingBundle{I}  \times \op{id}"] \arrow[d] \\
            \productCharts{K}  \arrow[u, bend left, "\kuranishiSection{K}"] & \productChartsIntersection{K}{J} \arrow[l, swap, "\breakingBase{I}\times \op{id}"] \arrow[u, bend left, "\kuranishiSection{J}"]
            \end{tikzcd}.       
        \end{equation}
       
        \item (Compatibility with lower strata)\label{P2} Suppose $$\iT{I} = (\underbrace{i_1,\dots, i_m}_{\iT{I'}}, \underbrace{i_{m+1} \dots, i_{n(\iT{I''})}}_{\iT{I''}}).$$
        When the gluing parameter $T_m \in [R, \infty)$ is sufficiently large, $\kuranishiSection{I}$ is independent of $T_{m}$ and equals $(\kuranishiSection{I'}, \kuranishiSection{I''}): \productCharts{I'}\times [R,\infty) \times \productCharts{I''} \to \productBundle{I'} \oplus \productBundle{I''}$.
        
    \end{enumerate}
\end{definition}
\begin{remark}
    In the case when $\iT{K} = (k_1, k_2)$ and $J = \contraction{K}{K} = c$, \ref{P1} is the same as \ref{sigma1}.
\end{remark}
For any critical points $p,q$, let $\ell$ be the integer such that $\M_\ell = \M(p,q)$.
We define the {\em perturbed moduli space} for $\M(p,q)$ to be $$\mathcal Z(p,q) = \mathcal Z_\ell  = \coprod_{\contraction{I}{I} = \ell} (\obstructionSection{I}+\sigma_{\iT{I}})/\sim,$$
where $\sim$ is given by \ref{k1}(f) and \ref{P1}.
\begin{lemma}[Perturbation section for $\Kur$]\label{lemma: 0 dim moduli space}
Given a compatible section $(\kuranishiSectionBold_1, \kuranishiSectionBold_2, \dots)$, suppose that it satisfies:
\begin{enumerate}[label = ($\bold P$\arabic*),start = 3]
    \item \label{P3} ($C^1$-small) $\sigma_I$ is $C^1$-small.
    \item \label{P4} (Transversality) $\kuranishiSection{I} + \obstructionSection{I}$ is transverse to the zero section.
\end{enumerate}
Then $\mathcal Z_\ell$ is a manifold of dimension $\op{virdim} \M_\ell$. Moreover,
\be
    \item When $\ind p - \ind q - 1 < 0$, the perturbed moduli space is empty;
    
    \item When $\ind p - \ind q - 1 = 0$, the perturbed moduli space $\mathcal Z(p,q)$ consists of finitely many oriented points;

    \item When $\ind p - \ind q - 1 = 1$, then the perturbed moduli space $\mathcal Z(p,q)$ is a compact, oriented, one-dimensional manifold with boundary diffeomorphic to $$\coprod_{r, \ind r = \ind q + 1}\mathcal Z(p,r) \times \mathcal Z(r,q)$$ through an orientation preserving diffeomorphism.
\ee
\end{lemma}
\begin{proof}
   Apart from the orientation, this lemma follows from \ref{P1}-\ref{P4}, and (f) of \ref{k1}. The orientation is explained in Section~\ref{section: morse homology}.
\end{proof}

With a perturbed moduli space, one can define the Morse homology in the standard way.

\section{Orientation and Morse homology}\label{section: morse homology}
The orientation of the moduli space in the obstructed case has been extensively studied in the context of pseudoholomorphic curves, and the Morse case is significantly simpler.

For any Fredholm operator $L: X \to Y$, we define $\det L = \wedge^{\op{top}} \ker L \otimes \wedge^{\op{top}} \op{coker} L^*.$
For each smooth map $u\in \widetilde{\mathcal B}$ as in Formula~\eqref{eqn: tilde B}, we denote by $\tilde{\mathfrak o}(u)$ a choice of  non-zero vector in $\det D_u$.
Let $\mathfrak o(u)$ be the non-zero vector in $\wedge^{\op{top}}(\ker D_u / \R\langle \partial_s \rangle ) \otimes \wedge^{\op{top}}(\op{coker} D_u)^*$ determined by 
\[
\tilde{\mathfrak o}(u) = \R\langle \partial_s \rangle \otimes \mathfrak o(u),
\]
where $s$ is the coordinate of the domain of $u$, and $\partial_s$ is the unique element in $\ker D_u$ given by translating the domain of $u$ in the positive direction. Let $\obstructionBundle \to \M$ be the obstruction bundle, and let $\sigma$ be a transverse section of $\obstructionBundle$. Let $\mathcal Z = \sigma^{-1}(0)$, and let $u \in \mathcal Z$ be an arbitrary element. Then 
\begin{align*}
\wedge^{\op{top}} T_u \mathcal Z & = \wedge^{\op{top}} \ker \nabla_u \sigma  \\
& \simeq \wedge^{\op{top}} T_u \M \otimes \wedge^{\op{top}} \obstructionBundle_u^* \\
& \simeq \wedge^{\op{top}}(\ker D_u / \R\langle \partial_s \rangle ) \otimes \wedge^{\op{top}}(\op{coker} D_u)^*  \ni \mathfrak o(u),
\end{align*}
where $\nabla_u \sigma$ is the derivative of $\sigma$ at $u$ projected to the fiber, and the last isomorphism comes from Lemma 9.2.1 in \cite{bao2018definition} (picking $E = \obstructionBundle_u$).
Therefore, $\mathfrak o(u)$ determines a choice of orientation of $T_u \mathcal Z$. This completes the orientation for the interior chart. For the boundary chart, a standard gluing argument (see for instance Section 9 of \cite{hutchings2009gluing},  \cite{bourgeois2004coherent} or \cite{bao2023coherent}) shows that the gluing map preserves the orientation, thus implying the statement of Lemma~\ref{lemma: 0 dim moduli space}.
In the case when $\dim \mathcal Z = 0$, $\mathfrak o(u) \in \{\pm 1\}$ for any $u \in \mathcal Z$.

For completeness, we provide a brief review of Morse homology. 
Let $R$ be a ring, and let $C_* = \oplus_k C_k(M)$, where $C_k(M)$ is the $R$-module freely generated by critical points of $f$ with index $= k$.
The differential $\partial: C_* \to C_{*-1}$ is the linear map defined on generators as 
\[
\partial [p] = \sum_{\substack{q \in \crit{f} \\ \ind q = \ind p - 1}} \sum_{u \in \mathcal Z(p,q)} \mathfrak o (u) [q],
\]
where $\mathfrak o(u) \in \{\pm 1\}$.

\begin{theorem}
    $\partial^2 = 0$.
\end{theorem}
\begin{proof}
    This follows directly from Lemma~\ref{lemma: 0 dim moduli space}.
\end{proof}
We define the Morse homology $H_*(M) = \ker \partial / \op{im} \partial$.
To show that Morse homology is an invariant of $M$, one can construct a family version of the minimal semi-global Kuranishi structure. 
Since the main purpose of this paper is to generalize obstruction bundle gluing to provide a computable theory, we skip the proof of invariance.

\section{Construction of minimal semi-global Kuranishi structures}\label{section: construction of semi-global}
In this section, we explain the construction of minimal semi-global Kuranishi structures and perturbation sections. For simplicity of notation, we use the following prototypical example:

\begin{example}[A prototypical example]\label{example: prototypical example}
    Let $p, q, r, s \in \crit(f)$ be the critical points of $f$ with the least actions: $\cdots > f(p) > f(q) > f(r) > f(s)$. 
    Let $\M_1 = \M(r,s)$, $\M_2 = \M(q,r)$, $\M_3 = \M(p,q)$, $\M_4 = \M(q,s)$, $\M_5 = \M(p,r)$, and $\M_6 = \M(p,s)$. In particular, $\M_1$, $\M_2$, and $\M_3$ are compact. See Figure~\ref{fig: moduli space label}.
\end{example}

\begin{proposition}
    Suppose the moduli spaces $\M_1, \M_2, \dots $ are cleanly cut out. Then there exists a minimal semi-global Kuranishi structure $(\kur_1, \kur_1, \dots )$ for $(\overline{\M}_1, \overline{\M}_2, \dots )$.
\end{proposition}

A key new ingredient in the definition/construction of a clean Kuranishi structure is \ref{k2}: iterated gluing equals simultaneous gluing. We achieve this by modifying the breaking map.

\begin{proof}
We will prove this for the prototypical Example~\ref{example: prototypical example}. The general case is only notationally more complex.  

A minimal semi-global Kuranishi structure for $(\overline{\M}_1, \overline{\M}_2, \overline{\M_3})$ consists of only one chart $V_i = \M_i$ for each $i = 1,2,3$. Thus, we have the obstruction bundles $\obstructionBundle_i \to \M_i$ and $\kur_i = \{\chartData{i} = (V_i, \breakingBase{i}, \breakingBundle{i}, \obstructionSection{i})\}$ with $\breakingBase{i}= \op{id}$, $\breakingBundle{i} = \op{id}$, and $\obstructionSection{i} = 0$ for $i = 1,2,3$.
\begin{figure}
    \centering
    \begin{tikzpicture}
        \draw[arrows = {-Latex[width=3pt, length=10pt]}](2,4) node[anchor = south]{$p$}--(0,3) node[near start, left]{$\M_3$};
        \draw[arrows = {-Latex[width=3pt, length=10pt]}](0,3) node[anchor = south]{$q$}--(0,1) node[midway, left]{$\M_2$};
        \draw[arrows = {-Latex[width=3pt, length=10pt]}](0,1) node[anchor = east]{$r$}--(2,0) node[anchor = north]{$s$}  node[midway, left]{$\M_1$};
        \draw[arrows = {-Latex[width=3pt, length=10pt]}, color = red](2,4)--(0,1) node[near start, left]{$\M_5$};
        \draw[arrows = {-Latex[width=3pt, length=10pt]}, color = blue](0,3)--(2,0) node[midway, right]{$\M_4$};
        \draw[arrows = {-Latex[width=3pt, length=10pt]}, color = orange](2,4)--(2,0) node[midway, right]{$\M_6$};
    \end{tikzpicture}
    \caption{Orders of moduli spaces}
    \label{fig: moduli space label}
\end{figure}
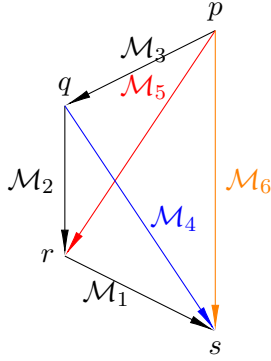

\begin{figure}
\centering
\begin{subfigure}[t]{.45\textwidth}
    \centering
    \begin{tikzpicture}[scale = 1]
        \def \e {0.2};
        \filldraw[blue!20, fill opacity=0.5]{(0,2)--(2,2)--(2,4)--(0,4)}; 
        \filldraw[red!20, fill opacity=0.5]{(2-2*\e,3)--(4,3)--(4,4)--(2-2*\e,4)}; 
        \filldraw[green!20, fill opacity=0.5]{(0,0)--(1,0)--(1,2+2*\e)--(0,2+2*\e)}; 
        \filldraw[yellow!20, fill opacity=0.5]{(1-\e,0)--(4,0)--(4,3+\e)--(2-\e,3+\e)--(2-\e,2+\e)--(1-\e,2+\e)}; 
        \draw (0.5, 1) node{$\lW_{43}$};
        \draw (3, 3.5) node{$\lW_{15}$};
        \draw (1,3) node{$\lW_{123}$};
        \draw (3,1) node{$\lW_{6}$};
        \draw[arrows = {-Latex[width=3pt, length=6pt]}] (4,4) -- (0,4) node[midway, above]{$\tau_{r}$} node[pos= 1, above]{$\infty$};
        \draw[arrows = {-Latex[width=3pt, length=6pt]}] (0,0) -- (0,4) node[midway, left]{$\tau_{q}$} node[pos= 1, left]{$\infty$};
    \end{tikzpicture}
    \caption{The projection of charts for $\M_6$. $\lW_{43} \cap \lW_{15} = \emptyset$.}
    \label{fig: M6 small}
\end{subfigure}\quad\quad
\begin{subfigure}[t]{.45\textwidth}
    \centering
    \includegraphics[scale = 0.3]{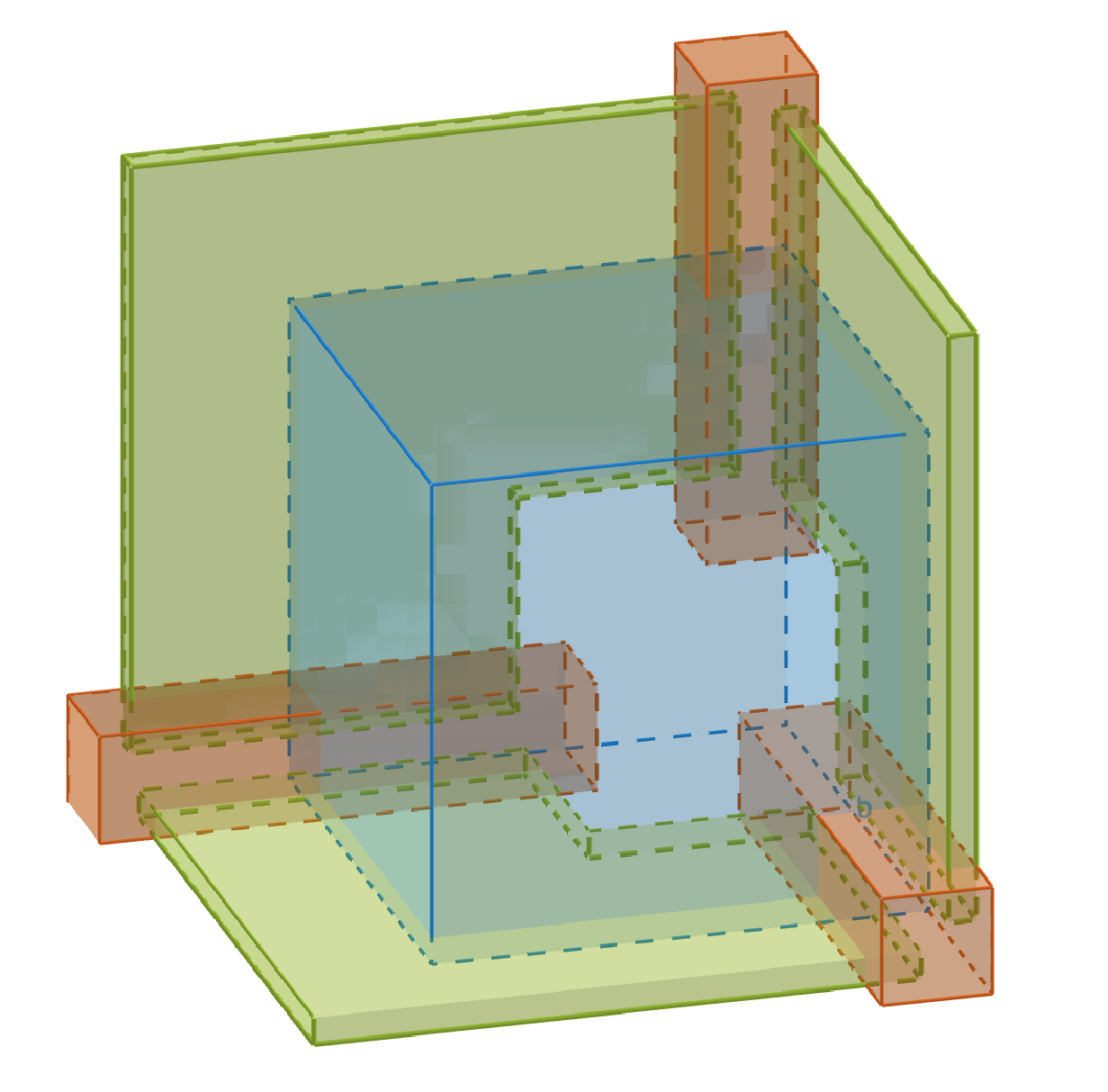}
    \caption{Charts for $\M_\ell$ for some $\ell$, where a smooth trajectory can degenerate into broken trajectories consisting of four pieces. For any $\iT{I}, \iT{J}$ such that $\contraction{I}{I} = \contraction{J}{J} = \ell$, we have $\lW_{I} \cap \lW_{J} = \empty$ if and only if $\iT{I}\leq \iT{J}$ or $\iT{J} \leq \iT{I}$.}
    \label{fig: M10}
\end{subfigure}
\caption{}
\end{figure}
    
We now construct $\kur_4 = \{\chartData{4}, \chartData{12}\}$. Let $\compactSubset_4$ be a large compact subset of $\M_4$, such that any element of $\M_4$ not in $\compactSubset_4$ is $\delta_{12}$-close to breaking into some broken curve in $\M_1 \times \M_2$ for some small fixed $\delta_{12} > 0$. 

Let $\lW_4$ be an open subset in $\mathcal B(q,s)$ that contains $\compactSubset_4$. Let $\lW_{12}$ be an open subset in $\mathcal B(q,s)$ that contains all curves in $\M_4$ that are $\delta_{12}$-close to breaking. Consider the sub-bundle $\obstructionBundle_{12}$ of $\mathcal E_4 = \mathcal E(q,s)$ over $\lW_{12}$ defined in \eqref{eqn: direct sum bundle}. 

We define $\obstructionBundle_4 \to \lW_4$ such that it is isomorphic to the obstruction bundle over $\lW_4$, and $\obstructionBundle_4$ is a sub-bundle of $\obstructionBundle_{12}$ when restricted to the overlap of $\lW_4 \cap \lW_{12}$. This can be achieved by interpolating the obstruction bundle and its orthogonal projection to $\obstructionBundle_{12}$, and shrinking $\lW_{12}$ to avoid the interpolation region while maintaining $\lW_{12} \cup \lW_4 \supset \M_4$. 

Let $\kChart{4} = \lW_4 \cap \M_4$ and $U_{12} = \morsedbar^{-1}(\obstructionBundle_{12})$. Note that $\morsedbar$ is transverse to $\obstructionBundle_{12}$ by Claim~\ref{claim: lambda surjective} and Remark~\ref{remark: surjectivity implies transverse}. By the Kuranishi gluing Claim~\ref{claim: kuranishi gluing}, there exists $R_{12} > 0$ such that we have the following diagram: 
\[
\begin{tikzcd}
    \obstructionBundle_1 \oplus \obstructionBundle_2 \arrow[r,"G_{12}^\sharp"] \arrow[d] & \obstructionBundle_{12} \arrow[d] \arrow[hookleftarrow]{r} & \obstructionBundle_4|_{V_{12}} \arrow[d] \\
    V_1 \times_{R_{12}} V_2 \arrow[r, "G_{12}"] & U_{12} \arrow[u, "\morsedbar", bend left] \arrow[hookleftarrow]{r} & \kChart{12}
\end{tikzcd},
\]
where $G_{12}$ is a diffeomorphism onto its image, 
$G_{12}^\sharp$ is a bundle isomorphism.
By choosing $\delta_{12}>0$ sufficiently small (and hence sufficiently large $\compactSubset_4$), the image of $G_{12}$ contains all curves that are $\delta_{12}$-close to breaking, and hence $\op{im}(G_{12}) \cup \compactSubset_{4} = \M_4$.  

Let $\kChart{12} = \M_4 \cap \op{im}(G_{12})$, $\obstructionSection{12} = (G_{12}^\sharp)^{-1} \circ \morsedbar \circ G_{12}$,
$\breakingBase{12} = G_{12}^{-1}|_{\kChart{12}}$, and $\breakingBundle{12} = (G_{12}^\sharp)^{-1}$. This gives the following diagram:
\[
\begin{tikzcd}
    \productBundle{12} \arrow[d] \arrow[hookleftarrow]{r}{\breakingBundle{12}} & \obstructionBundle_4|_{\kChart{12}} \arrow[d] \\
    \productCharts{12} \arrow[u, bend left, "\obstructionSection{12}"] \arrow[hookleftarrow]{r}{\breakingBase{12}} & \kChart{12}
\end{tikzcd},
\]
where we recall $\productCharts{12} = V_1 \times_{R_{12}} V_2$.

Similarly, $\obstructionBundle_5 \to \M_5$ and $\kur_5 = \{\chartData{5}, \chartData{23}\}$ are constructed. 

Now we construct $\kur_6 = \{\chartData{6}, \chartData{43}, \chartData{15}, \chartData{123}\}$.

First, we will construct open subsets $\lW_6, \lW_{43}, \lW_{15}, \lW_{123} \subset \mathcal B(p,s)$, inside which we will find $\kChart{6}$, $\kChart{43}$, $\kChart{15}$, and $\kChart{123}$, respectively. They need to satisfy the following conditions:
\be
    \item $\cup_{\contraction{I}{I} = 6} \lW_{\iT{I}} \supset \M_6$,
    \item $\lW_{43} \cap \lW_{15} = \emptyset$.
\ee 

To achieve this, fix a small constant $\epsilon > 0$. For any $u \in \M_6$, let $\tau_{r}(u)$ be the amount of time, called travel time,  that $u$ stays inside the ball of radius $\epsilon$ centered at the critical point $r$. 
Similarly, let $\tau_{q}(u)$ be the amount of time that $u$ stays inside the ball of radius $\epsilon$ centered at the critical point $q$.
Let $\lW$ be a small neighborhood of $\M_6$ in $\mathcal B(p,s)$, such that the projection map $\Pi: \lW \to \R^2$ mapping $u$ to $(\tau_q(u), \tau_r(u))$ is defined. Then, we can choose open subsets $\lW_{\iT{I}}$ by taking the pre-image of open subsets in $\R^2$ as indicated in Figure~\ref{fig: M6 small} for $\iT{I} \in  \{6, 43, 15, 123\}$. The requirement for the open subsets is that: for any $\iT{I} \in \{43, 15, 123\}$, $\lW_{\iT{I}}$ is sufficiently small so that the breaking map $\breakingBase{I}$ will be defined over $\lW_{\iT{I}} \cap \M_6$. This is guaranteed by:
$\Pi (\lW_{\iT{I}} \cap \M_6) \subset \op{im}(\Pi \circ f_{\iT{I}})$, where $f_I$ is the pregluing map, defined as in the proof of Theorem~\ref{thm: gluing} over the following regions: $\compactSubset_4 \times_{R_{43}} \compactSubset_3$, $\compactSubset_1 \times_{R_{15}} \compactSubset_5$, and $\compactSubset_1 \times_{R_{123}} \compactSubset_2 \times_{R_{123}} \compactSubset_3$ for $\iT{I} = 43, 15,$ and $123$, respectively. 
Here $R_{\iT{I}}$ are large real numbers so that the pre-gluing maps $f_{\iT{I}}$ are defined, and we choose $R_{43} \gg R_{123}$.
(For moduli spaces of higher energy that can break into a broken curve consisting of four trajectories, which is the general case outside of the prototypical example, the charts are given by pre-images of the open subsets of $\R^3$ shown in Figure~\ref{fig: M10}.)

Over the image of the pregluing map $f_{123}$, we define a sub-bundle $\obstructionBundle_{123} \subset \mathcal E|_{\op{im}(f_{123})}$ such that $f_{123}^* \obstructionBundle_{123}$ is isomorphic to $\oplus_{i=1}^3 \op{pr}_i^* \obstructionBundle_i$, where $\op{pr}_i: \compactSubset_1 \times_{R_{123}} \compactSubset_2 \times_{R_{123}} \compactSubset_3 \to \compactSubset_i$ is the projection map. We then extend this bundle over $\lW_{123}$ by parallel transport. Next, we define $U_{123} = \morsedbar^{-1} (\obstructionBundle_{123})$ and $V_{123} = U_{123} \cap \M_6$. By simultaneous gluing, we obtain the upper left parallelogram (the one involving $\breakingBase{123}$ and $\breakingBundle{123}$) in the Diagram~\eqref{eqn: main construction diagram}:

\begin{figure*}[t] 
\begin{equation}\label{eqn: main construction diagram}
\begin{tikzcd}[scale cd = 0.8]
    & & \obstructionBundle_{123} \arrow[d]\arrow[ldd, "\breakingBundle{123}", swap, sloped] \arrow[hookleftarrow]{rdd} & \\
    & & U_{123} \arrow[ldd, "\breakingBase{123}", swap, sloped] \arrow[hookleftarrow]{rdd} \arrow[u, bend left, "\morsedbar"] & \\
    &\productBundle{123} \arrow[d]&   & \obstructionBundle_{43} \arrow[d] \arrow[dr] \arrow[ldd] & \\
    &\productCharts{123} &  &  U_{123} \cap U_{43}  \arrow[hook, dr] \arrow[ldd] & \obstructionBundle_{43} \arrow[d] \arrow[ldd, "\breakingBundle{43}"] \\
    \obstructionBundle_{12}\oplus \obstructionBundle_3 \arrow[d] \arrow[hookleftarrow]{rr} \arrow[ruu, "\breakingBundle{12}\times \op{id}", sloped] &  & \productBundle{43} \arrow[d] \arrow[luu] \arrow[dr] &    &\hspace{1cm} U_{43} \hspace{1cm} \arrow[ldd, "\breakingBase{43}"] \arrow[u, bend right, "\morsedbar", swap]\\
    U_{12}\times_{R_{123}} V_3 \arrow[ruu, "\breakingBase{12}\times \op{id}", sloped, swap ] \arrow[hookleftarrow]{rr} &  & (U_{12}\cap\kChart{4})\times_{R_{43}} \kChart{3} \arrow[luu] \arrow[dr, hook]  & \productBundle{43} \arrow[d] &     \\
    &  &   & \productCharts{43} &   
\end{tikzcd}        
\end{equation}
\end{figure*}
We define the bundle $\obstructionBundle_{43} \to U_{43}$ with the additional requirement that $\obstructionBundle_{43}$ is a sub-bundle of $\obstructionBundle_{123}$ when restricted to $U_{43} \cap U_{123}$.
First, we construct the bundle $\obstructionBundle_{43}$ over $\lW_{43}$ as in the proof of Theorem~\ref{thm: gluing} with $\compactSubset_1 = \M_1$ and $\compactSubset_4$ as defined in the previous steps. Over $\lW_{123} \cap \lW_{43}$, we have two bundles: $\obstructionBundle_{43}$ and $P(\obstructionBundle_{43})$, where $P: \obstructionBundle_{43} \to \obstructionBundle_{123}$ is the orthogonal projection. Note that $\obstructionBundle_{43}$ and $P(\obstructionBundle_{43})$ are arbitrarily close as the travel time $\tau_{q}$ goes to $+\infty$. 
We then shrink $\lW_{123}$ slightly and replace $\obstructionBundle_{43}$ with $P(\obstructionBundle_{43})$ over $\lW_{123} \cap \lW_{43}$, extending it smoothly to the rest of $\lW_{43}$. This completes the construction of $\obstructionBundle_{43}$, ensuring it is a sub-bundle of $\obstructionBundle_{123}$ over $\lW_{123} \cap \lW_{43}$.
Next, we define $U_{43}= \morsedbar^{-1}(\obstructionBundle_{43})$. 
This yields the second parallelogram at the top of Diagram~\ref{eqn: main construction diagram} (the one that involves $\obstructionBundle_{43}\to \obstructionBundle_{123}$).
The inverse of the gluing process provides the parallelogram involving $(\breakingBase{43}, \breakingBundle{43})$ on the lower right of the diagram. 

We now define $\obstructionBundle_6$. 
Let $\compactSubset_6 \subset \M_6 \cap \lW_{6}$ be a large compact subset such that $\compactSubset_6 \cup V_{123} \cup V_{43} \cup V_{15} =\M_6$, and $\compactSubset_6 \subset V_6 \subset \M_6$ be an open subset.
First, we define $\obstructionBundle_6$ over $\kChart{6}$ ensuring it is a sub-bundle of $\obstructionBundle_{123}$ over $\kChart{6} \cap \kChart{123}$. 
Then, we project $\obstructionBundle_6$ to $\obstructionBundle_{43}$ over $\kChart{6} \cap \kChart{43}$, which does not affect the sub-bundle property $\obstructionBundle_6 \subset \obstructionBundle_{123}$ over $\kChart{6} \cap \kChart{123}$ because $\obstructionBundle_{43} \subset \obstructionBundle_{123}$ over $\kChart{43} \cap \kChart{123}$. 

Starting from $U_{123} \cap U_{43}$, there are two ways to reach $\productCharts{123}$: via $\breakingBase{123}$ or via $(\breakingBase{12} \times \op{id}) \circ \breakingBase{43}$, ignoring the inclusion maps. We now modify $\breakingBase{123}$ so that 
\begin{equation} \label{eqn: base commute}
    \breakingBase{123} = (\breakingBase{12} \times \op{id}) \circ \breakingBase{43},
\end{equation}
and further, we modify $\breakingBundle{123}$ so that the associated bundle maps satisfy:
\begin{equation} \label{eqn: bundle commute}
    \breakingBundle{123} = (\breakingBundle{12} \times \op{id}) \circ \breakingBase{43}.
\end{equation}
This can be achieved as follows: we modify the map $\breakingBase{123}$ in a tubular neighborhood of $U_{123} \cap U_{43}$ so that Equation~\eqref{eqn: base commute} holds. The explicit construction uses smooth functions $0 \leq \rho_1, \rho_2 \leq 1$ defined over $U_{123}$ such that $\rho_1 + \rho_2 = 1$ and $\rho_1 = 1$ on $U_{123} \cap U_{43}$. Then, the modified $\breakingBase{123}$ can be defined by taking the center of mass of $\breakingBase{123}$ and $(\breakingBase{12} \times \op{id}) \circ \breakingBase{43}$ using the weights $\rho_1$ and $\rho_2$ (see, for example, \cite{peters1984cheeger} for references on the notion of center of mass in Riemannian geometry).
Similarly, we modify $\breakingBundle{123}$ by taking a linear interpolation, and get Equation~\eqref{eqn: bundle commute}.

We define the obstruction section of $\productBundle{123} \to \productCharts{123}$ by $\obstructionSectionS_{123} = \breakingBundle{123} \circ \morsedbar \circ \breakingBase{123}^{-1}$, and the obstruction section of $\productBundle{43} \to \productCharts{43}$ by $\obstructionSectionS_{43} = \breakingBundle{43} \circ \morsedbar \circ \breakingBase{43}^{-1}$.

By Equations~\eqref{eqn: base commute} and ~\eqref{eqn: bundle commute}, we have 
\begin{equation}
    (\breakingBundle{12} \times \op{id}) \circ \obstructionSection{43} = \obstructionSection{123} \circ (\breakingBase{12} \times \op{id}).
\end{equation}

The obstruction bundle $\obstructionBundle_{15} \to \kChart{15}$ and obstruction section $\obstructionSectionS_{15}$ can be defined similarly. The modification of $(\breakingBase{123}, \breakingBundle{123})$ does not interfere with the modification for $\obstructionBundle_{43} \to \kChart{43}$ because we choose $\lW_{15} \cap \lW_{43} = \emptyset$.

\end{proof}

\begin{proposition}
    Suppose the moduli spaces $\M_1, \M_2, \dots $ are cleanly cut out. Given a minimal semi-global Kuranishi structure $(\Kur_1, \Kur_2, \dots)$, there exists a perturbation section.
\end{proposition}
\begin{proof}
    We prove this for the prototypical Example~\ref{example: prototypical example}. The general case is a straightforward generalization with more complex notation.
    The minimal semi-global Kuranishi structure for $\overline{\M}_i$ consists of only one chart $V_i = \M_i$ for $i = 1,2,3$. 
    We choose sections $\sigma_i: \M_i \to \obstructionBundle_i$ that are transverse to the zero sections.

    The minimal semi-global Kuranishi chart for $\overline{\M}_4$ consists of two charts $V_4$ and $V_{12}$; the minimal semi-global Kuranishi chart for $\overline{\M}_5$ consists of two charts $V_5$ and $V_{23}$.
    The minimal semi-global Kuranishi chart for $\overline{\M}_6$ consists of four charts $V_6, V_{15}, V_{43}, V_{123}$.

    Now we construct perturbation sections for $\M_4$, $\M_5$, and $\M_6$:
    
    \noindent{} ($\bs{\overline{\M}_5}$). First, we construct a perturbation section for $\M_5$. This part is essentially the same as the construction in Section~\ref{section: perturbation section}.
    We choose a section $\sigma_5: V_5 \to \obstructionBundle_5$ that is transverse to the zero section.
    Next, we choose a section $\sigma_{23}: \productCharts{23} \to \productBundle{23}$ as follows (recall $\V_{23} = V_2 \times_{R_{23}} V_3$ and $\productBundle{23} = \obstructionBundle_2 \oplus \obstructionBundle_3$): 
    
    Choose a small constant $\epsilon_1 > 0$.
    \begin{enumerate}[label=($\sigma_{23}$-\alph*)]
        \item Over the region $\{R_{23} \leq T \leq R_{23} + \epsilon_1 \}$, we choose a tubular neighborhood of $\op{im}(\breakingBase{23})$ with coordinates $(x_5,y_{23}^5)$, 
        where $x_5$ is the coordinate of $\op{im}(\breakingBase{23})$ and $y_{23}^5$ is the coordinate for the normal direction,
        and define a section $\sigma_{23}(x_5,y_{23}^5) = (\beta_{23}^5(y_{23}^5)\sigma_5(x_5), 0) \in \obstructionBundle_5 \oplus \obstructionBundle_5^\perp = \productBundle{23}.$

        \item Choose $R_{23}'\gg R_{23}$. Over the region $\{T \geq R_{23}'\}$, we define $\sigma_{23} = (\sigma_2, \sigma_3)$.

        \item Over the region $\{ R_{23} + \epsilon_1 \leq T \leq R_{23}' \}$, we define $\sigma_{23}$ to be an arbitrary section such that $\sigma_{23} + \obstructionSectionS_{23}$ is transverse to the zero section.

    \end{enumerate}
    
    \noindent{}($\bs{\overline{\M}_4}$). The sections $\sigma_4$ and $\sigma_{12}$ are constructed in the same way. 
    
    \noindent{}($\bs{\overline{\M}_6}$). We choose a section $\sigma_6: V_6 \to \obstructionBundle_6$ that is transverse to the zero section.
    Now we define sections $\sigma_{43}$, $\sigma_{15}$, and $\sigma_{123}$:
    
    For $\sigma_{15}: \productCharts{15} \to \productBundle{15}$, choose $\epsilon_2$ such that $0 < \epsilon_2 < \epsilon_1$ and increase $R'_{12}$, if necessary, such that $R'_{12} \gg R_{15} \gg R_{12}$.
    \begin{enumerate}[label=($\sigma_{15}$-\alph*)]
        \item \label{sigma15a} Over the region $\{ R_{15} \leq T_r \leq R_{15} + \epsilon_2\}$ (see Figure~\ref{fig: M6}), we choose a tubular neighborhood of $\op{im}(\breakingBase{15})$ with coordinates $(x_6, y_{15}^6)$, and define the section $\sigma_{15} (x_6, y_{15}^6) = (\beta_{15}^6(y_{15})\sigma_6(x_6), 0)$.
        
        \item Over the region $\{ T_r \geq R_{12}'\}$, we define $\sigma_{15} = (\sigma_1, \sigma_5)$. 

        \item Over the rest of $\productCharts{15}$, we define $\sigma_{15}$ to be an arbitrary section such that $\sigma_{15} + \obstructionSectionS_{15}$ is transverse to the zero section.
    \end{enumerate} 

    The section $\sigma_{43}$ can be defined similarly.

    For $\sigma_{123}$: 
    \begin{enumerate}[label=($\sigma_{123}$-\alph*)]
        \item \label{sigma123a} Over the region $\{R_{23} \leq T_q \leq R_{43} + \epsilon_2, R_{12} \leq T_r \leq R_{12} + \epsilon_2\} \cup \{R_{12} \leq T_r \leq R_{15} + \epsilon_2, R_{23} \leq T_q \leq R_{23} + \epsilon_2\}$, shown as the ``intersection" of $\V_{123}$ and $\V_6$ in Figure~\ref{fig: M6}, we choose a tubular neighborhood of $\op{im}(B_{123})$ with coordinates $(x_6, y_{123}^6)$, 
        and define $$\sigma_{123}(x_6, y_{123}^6) = (\beta_{123}^6(y_{123})\sigma_6(x_6), 0)\in \obstructionBundle_6 \oplus \obstructionBundle_6^\perp = \productBundle{123}.$$

        \item \label{sigma123b} Over the region $\{R_{23} \leq T_q \leq R_{23} + \epsilon_1, T_r \geq R_{15}\}$, the ``intersection" between $\V_{15}$ and $\V_{123}$, we choose a tubular neighborhood of $\op{im}(\op{id} \times B_{23})$ with coordinate $(x_{15}, y_{123}^{15})$ (recall $\op{id}\times B_{23}$ maps a subset of $V_{15} = V_1 \times_{R_{15}} V_5$ into $V_1 \times_{R_{123}} V_2 \times_{R_{123}} V_3 = \productCharts{123}$) and define the section 
        $$\sigma_{123}(x_{15}, y_{123}^{15}) = (\beta_{123}^{15}(y_{123}^{15})\sigma_{15}(x_{15}),0) \in \productBundle{15} \oplus \productBundle{15}^\perp  = \productBundle{123}.$$

        \item \label{sigma123c} Over the region $\{R_{12} \leq T_r \leq R_{12} + \epsilon_1, T_q \geq R_{43}\}$, the ``intersection" between $\productCharts{43}$ and $\V_{123}$, we choose a tubular neighborhood of $\op{im}(\breakingBase{12}\times \op{id})$ with coordinates $(x_{43}, y_{123}^{43})$ and define 
        $$\sigma_{123}^{43}(x_{43}, y_{123}^{43}) = (\beta_{123}^{43}(y_{123}^{43})\sigma_{43}(x_{43}), 0)\in \productBundle{43} \oplus \productBundle{43}^\perp = \productBundle{123}.$$

        \item \label{sigma123d} Over the region $\{T_r \geq R_{12}'\}$, we choose $\sigma_{123} = (\sigma_1, \sigma_{23})$.

        \item \label{sigma123e} Over the region $\{T_q \geq R'_{23}\}$, we choose $\sigma_{123} = (\sigma_{12}, \sigma_3)$.

        \item \label{sigma123f} Over the rest of $\productCharts{123}$, we define $\sigma_{123}$ arbitrarily so that $\sigma_{123} + \obstructionSectionS_{123}$ is transverse to the zero section.

    \end{enumerate}

    	Compatibility between the above \ref{sigma123a}, \ref{sigma123b}, and \ref{sigma15a}: in their overlap, 
	$\{R_{15} \leq T_r \leq R_{15} + \epsilon_2, R_{23} \leq T_q \leq R_{23} + \epsilon_2\}$ (marked as $X$ in Figure~\ref{fig: M6}),
	from Definitions \ref{sigma123b} and \ref{sigma15a} and $x_{15} = (x_6, y_{15}^6)$, we get 
	$$\sigma_{123}(x_{15}, y_{123}^{15}) = (\beta_{123}^{15}(y_{123}^{15})\sigma_{15}(x_{15}),0) = (\beta_{123}^{15}(y_{123}^{15})\beta_{15}^6(y_{15}^6)\sigma_6(x_6), 0).$$
	Now we compare with Definition~\ref{sigma123a}.
	Since iterated gluing = simultaneous gluing, we have $(\breakingBase{23}\times \op{id}) \circ \breakingBase{15} = \breakingBase{123}$.
	Hence, we can choose coordinates for $\productCharts{123}$ in that region, so that $y_{123}^6 = (y_{15}^6, y_{123}^{15})$ and thus require that the choice of $\beta_{123}^6$ satisfies $\beta_{123}^6(y_{15}^6, y_{123}^{15}) = \beta_{123}^{15}(y_{123}^{15})\beta_{15}^6(y_{15}^6)$.

	Compatibility between \ref{sigma123d} and \ref{sigma123e}: over their overlapping region, $\sigma_{23} = (\sigma_2, \sigma_3)$ and $\sigma_{12} = (\sigma_1, \sigma_2).$

	Now over each chart $\V_{\iT{I}}$, we have an obstruction section $\obstructionSectionS_{\iT{I}}$. 
	By construction, the obstruction sections satisfy the following compatibility conditions and automatically patch together:
	\begin{enumerate}
	    \item $\obstructionSectionS_{i} = 0,$ for $i = 1, 2, \dots, 6$.
	    \item $\obstructionSection{12} \circ \breakingBase{12}=0$.
	    \item $\obstructionSection{23} \circ \breakingBase{23}=0$.
	    \item $\obstructionSectionS_{15} \circ \breakingBase{15} = 0$.
	    \item $\obstructionSectionS_{43} \circ \breakingBase{43} = 0$.
	    \item $\obstructionSectionS_{123} \circ (\op{id} \times \breakingBase{23}) = (\op{id}\times \breakingBundle{23}) \circ \obstructionSectionS_{15}$.
	    \item $\obstructionSectionS_{123} \circ (\breakingBase{12} \times \op{id}) = (\breakingBundle{23} \times \op{id}) \circ \obstructionSectionS_{43}$.
	\end{enumerate}

	Now we check \ref{P3} and \ref{P4} for $\breakingBase{15}$, $\breakingBase{123}$, and $\op{id}\times \breakingBase{23}$.  
	For $\breakingBase{15}$ and $\breakingBase{123}$, as in the construction in Section~\ref{section: construction of a perturbation section}, \ref{P3} and \ref{P4} can be achieved by making $\sigma_6$ small in $C^1$-norm.
	For the map $\op{id} \times \breakingBase{23}$, 
	we extend the section $(\op{id} \times \breakingBase{23}) \circ (\sigma_{15} + \obstructionSectionS_{15})$ defined in $\V_{15}$ to the section $\sigma_{123} + \obstructionSectionS_{123}$ defined in $\V_{123}$.
	Consider the region $\mathcal D = \{R_{23} \leq T_q \leq R_{23} +\epsilon_2, T_r \geq R_{15}\}$. 
	Note that
	\begin{enumerate}
	    \item $\obstructionSectionS_{123}$ converges to $(0, \obstructionSectionS_{23})$ in $C^1$ as $T_r \to \infty.$ 
	    \item $\obstructionSectionS_{23}$ is transverse to $\breakingBundle{23}(\obstructionBundle_5)$ inside $\productBundle{23}$.
	\end{enumerate}
	Therefore, $\obstructionSectionS_{123}$ is transverse to $(\op{id}\times \breakingBundle{23})(\obstructionBundle_1 \oplus \obstructionBundle_5)$ inside $\productBundle{123}$ over $\mathcal D$.
	Finally, by choosing $\sigma_{15}$ small, for which we may have to re-choose $\sigma_1$ and $\sigma_5$, we can choose $\sigma_{123}$ small in $C^1$-norm. This implies \ref{P3} and \ref{P4}.    

	Therefore, $\kuranishiSection = \{\kuranishiSection{I}\}_I$ forms a perturbation section.

\end{proof}

\begin{figure}
    \centering
    \begin{tikzpicture}
        \draw[->, thick, red, opacity=0.5](5,0) -- (0,0);
        \draw[blue, thick, opacity=0.5] (10,0) -- (4,0);
        \draw (7,-0.5) node{\tiny $\V_5$};
        \draw (2, -0.5) node{\tiny $\V_{23}$};
        \filldraw[black] (5,0) circle (1pt) node[anchor=south]{\tiny $R_{23}$};
        \filldraw[black] (4, 0) circle (1pt) node[anchor=south]{\tiny $R_{23} + \epsilon_1$};
        \filldraw[black] (2, 0) circle (1pt) node[anchor=south]{\tiny $R'_{23}$};
    \end{tikzpicture}
    \caption{Chart for $\M_5$}
    \label{fig:chart for M5}
\end{figure}

\begin{figure}
    \centering
    \begin{tikzpicture}[scale = 0.65]
        \filldraw[black!5, fill opacity=0.5]{(0,0)--(10,0)--(10,10)--(0,10)}; 
        \filldraw[blue!20, fill opacity=0.5]{(0,5)--(5,5)--(5,10)--(0,10)}; 
        \filldraw[green!20, fill opacity=0.5]{(0,0)--(1,0)--(1,6.5)--(0,6.5)}; 
        \filldraw[red!20, fill opacity=0.5]{(3.5,9)--(10,9)--(10,10)--(3.5,10)}; 
        \def \e {0.4};
        \filldraw[yellow!20, fill opacity=0.5]{(1-\e,0)--(10,0)--(10,9+\e)--(5-\e,9+\e)--(5-\e, 5+\e)--(1-\e, 5+\e)};
        \draw (0,10) node[anchor=east]{\tiny $\infty$};
        \draw (0,10) node[anchor=west, rotate = 90]{\tiny $\infty$};
        \draw (8, 2) node{$\V_6$};
        \draw (3, 7) node{$\V_{123}$};
        \draw (0.5, 4) node{$\V_{43}$};
        \draw (6, 9.5) node{$\V_{15}$};
        \draw (4.8, 9.2) node{\tiny $X$};
        \draw[arrows = {-Latex[width=3pt, length=6pt]}] (10,10) -- (0,10);
        \draw (7, 10) node[anchor = west, rotate = 90]{$T_q$};
        \draw[arrows = {-Latex[width=3pt, length=6pt]}] (0,0) -- (0,10);
        \filldraw[black] (0,5) circle (1pt) node[anchor=east]{\tiny $R_{12}$};
        \filldraw[black] (0,6.5) circle (1pt) node[anchor=east]{\tiny $R_{12} + \epsilon_1$};
        \filldraw[black] (0,5+\e) circle (1pt) node[anchor=east]{\tiny $R_{12}+\epsilon_2$};
        \filldraw[black] (5, 10) circle (1pt) node[anchor=west, rotate = 90]{\tiny $R_{23}$};
        \filldraw[black] (3.5, 10) circle (1pt) node[anchor=west,rotate=90]{\tiny $R_{23}+\epsilon_1$};
        \filldraw[black] (5-\e, 10) circle (1pt) node[anchor=west, rotate = 90]{\tiny $R_{23}+\epsilon_2$};
        \draw (0,3) node[anchor = east]{$T_r$};
        \filldraw[black] (1,10) circle (1pt) node[anchor=west, rotate = 90]{\tiny $R_{43}$};
        \filldraw[black] (0,9) circle (1pt) node[anchor=east]{\tiny $R_{15}$};
    \end{tikzpicture}
    \caption{A corner chart for $\M_6$}
    \label{fig: M6}
\end{figure}
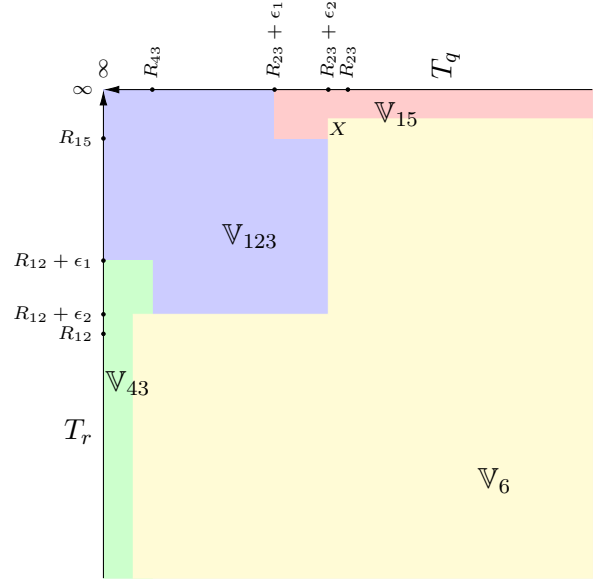


\printbibliography
\end{document}